\documentclass[pdflatex,sn-basic,Numbered,letterpaper]{sn-jnl}

\usepackage{graphicx}%
\usepackage{multirow}%
\usepackage{amsmath,amssymb,amsfonts}%
\usepackage{amsthm}%
\usepackage[title]{appendix}%
\usepackage{xcolor}%
\usepackage{textcomp}%
\usepackage{manyfoot}%
\usepackage{booktabs}%
\usepackage{algorithm}%
\usepackage{algorithmicx}%
\usepackage{algpseudocode}%
\usepackage{listings}%
\usepackage{mathrsfs}
\theoremstyle{thmstyleone}%
\newtheorem{theorem}{Theorem}[section]
\newtheorem{proposition}[theorem]{Proposition}%
\newtheorem{lemma}[theorem]{Lemma}%
\newtheorem{corollary}[theorem]{Corollary}

\theoremstyle{thmstyletwo}%
\newtheorem{remark}{Remark}%

\theoremstyle{thmstylethree}%
\newtheorem{definition}{Definition}%

\renewcommand{\div}{\mathrm{div}}
\providecommand{\norm}[1]{\left\| #1 \right\|}
\newcommand{\e}{\mathsf{a}}
\newcommand{\R}{\mathbb{R}}
\newcommand{\N}{\mathbb{N}}

\begin{document}

\title[Trace regularity]{Trace regularity of solutions to the Navier equations}


\author[1]{\fnm{Jerin Tasnim} \sur{Farin}}\email{21jtf2@queensu.ca}

\author*[2]{\fnm{Giusy} \sur{Mazzone}}\email{giusy.mazzone@queensu.ca}


\affil*[1,2]{\orgdiv{Department of Mathematics and Statistics}, \orgname{Queen's University}, \orgaddress{\street{48 University Ave}, \city{Kingston}, \postcode{K7L 3N6}, \state{Ontario}, \country{Canada}}}



%
\abstract{
 We present results on the trace regularity of the stress vector on the boundary of an elastic solid satisfying the time-dependent, displacement-traction problem for the Navier equations of linear elasticity in a bounded domain of $\R^3$. Specifically, the solid's displacement is subject to Dirichlet- and Neumann-type conditions on different portions of its boundary and possibly non-zero body forces and initial data. Our regularity results are reminiscent of the so-called ``hidden trace regularity'' results for solutions to the scalar wave equation obtained in \cite{LLT-86}.  
}




\maketitle

\section{Introduction}\label{sec:intro}
The Navier equations of linear elasticity provide a mathematical model for the Lagrangian description of the deformation of a linear elastic isotropic material. The constitutive equation for this type of material is given by the following generalized {\em Hooke’s law} for the {\em first Piola--Kirchhoﬀ} stress in terms of the displacement field $u$:
\begin{equation}\label{eq:piola-stress}
P(u):=2\mu e(u)+\lambda \mathrm{div}u,\qquad e(u):=\frac 12 (\nabla u+(\nabla u)^T),
\end{equation}
where $\mu$ and $\lambda$ are positive constants, called {\em Lam\'e constants}, and $e(u)$ is the (Lagrangian) {\em infinitesimal strain tensor}. The vector field $P(u)\cdot n$ is called {\em stress vector}, and it gives a measure of the 
surface force per unit 
area acting on a material surface whose unit outward normal is $n$. In this paper, we investigate the regularity of the stress vector $P(u)\cdot n$ on the boundary of a linear elastic isotropic material, along various types of solutions to the Navier equations. The importance of our results is two-fold. First, our trace regularity results do not immediately follow from the standard energy balances satisfied by the solutions to our differential equation. This is reminiscent of the well-known results in \cite{LLT-86} concerning the so-called ``hidden trace regularity'' for the normal gradient of solutions to the scalar wave equation. Secondly, we focus on the stress vector $P(u)\cdot n$ rather than the normal gradient $\nabla u\cdot n$. While this choice complicates our analysis, the stress vector has a well defined physical meaning, thus making our results more relevant from the physical point of view and applicable to other problems in continuum mechanics. 

Our analysis focuses on the following initial-boundary value problem for the Navier equations of linear elasticity in terms of the (Lagrangian) displacement filed $u$: 
\begin{equation}\label{eq:Navier-intro}
    \begin{aligned}
            &\frac{\partial^{2}u}{\partial t^2} =\mathrm{div} P(u) + F  &&\qquad \text{in } \Omega \times (0,T), \\
            &u(\cdot,0) =u_{0} &&\qquad \text{in } \Omega ,\\
            &\frac{\partial u}{\partial t}(\cdot,0) = u_{1} &&\qquad  \text{in } \Omega, \\
            &u =g && \qquad \text{on } \Gamma_{0} \times (0,T),\\
            & P(u)\cdot n =0 && \qquad \text{on } \Gamma_{1} \times (0,T). 
    \end{aligned}
\end{equation}
In the above, we assume that the solid occupies the bounded domain $\Omega:=\mathcal S\setminus\overline{\mathcal C}$ where $\mathcal S$ and $\mathcal C\subset\R^3$ are bounded domains with smooth boundaries\footnote{This assumption could be relaxed; The degree of smoothness of the boundary would change depending on the smoothness required to the solutions to our equations.} $\Gamma_1:=\partial \mathcal S$ and $\Gamma_0:=\partial\mathcal C$, respectively. An important assumption in our work is that the surface area of $\Gamma_0$ is strictly positive \footnote{As a matter of fact, one could switch the roles of $\Gamma_0$ and $\Gamma_1$, and assume that the surface area of $\Gamma_1$ is zero. All of our results would still hold. }. Also, we recall that the tensor $P$ has been defined in \eqref{eq:piola-stress}. The vector fields $F$, $u_0$, $u_1$, and $g$ denote the external body force, initial displacement, initial velocity, and the prescribed boundary displacement, respectively. Without loss of generality, we have set the solid's density to be equal to one. From the physical point of view, the first equation in \eqref{eq:Navier-intro} represents the balance of linear momentum for the solid. The fact that the first Piola--Kirchhoﬀ is symmetric implies that also the balance of angular momentum is satisfied. The boundary conditions mean that we prescribe the displacement on $\Gamma_0$, while the exterior boundary $\Gamma_1$ is {\em stress-free}, i.e., no external force is applied there. From the mathematical point of view, the first equation in \eqref{eq:Navier-intro} is a hyperbolic partial differential equation 
{\em not in divergence form}, in that it can not be written as a hyperbolic evolution equation $\displaystyle \frac{du}{dt}=Au+F$ on a suitable Banach space, 
with $A$ an elliptic linear operator, with smooth coefficients $a_{ij}$, of the form  
\[
Au=\sum^3_{i,j=1}\frac{\partial}{\partial x_i}\left(a_{ij}(x,t)\frac{\partial u_k}{\partial x_j}\right)\mathsf{a}_k,
\]
with respect to an orthonormal basis $\{\mathsf{a}_1,\mathsf{a}_2,\mathsf{a}_3\}$ (associated to a Cartesian coordinate system). We further note that our equation can not be even considered as a ``lower-order perturbation'' of a system of wave equations (with constant wave speeds). As a matter of fact 
\[
\div P(u)=\mu\Delta u+(\lambda+\mu)\nabla \div u. 
\]
Concerning the boundary conditions, \eqref{eq:Navier-intro}$_4$ is a {\em Dirichlet boundary condition}, while \eqref{eq:Navier-intro}$_5$ is a {\em Neumann-type boundary condition} as we prescribe only a combination of the components of the gradient of $u$ on the outer boundary $\Gamma_1$, and not the full normal gradient. 

Our main trace regularity results for the stress vector $P(u)\cdot n$ on the boundary $\Gamma_0$ are stated in Theorem \ref{thm trace regularity with nonzero B.C.} (for weak solutions to \eqref{eq:Navier-intro} in the sense of Remark \ref{rm:distributional-formulation}), Theorem \ref{trace regularity from adjoint weak solution} (for solutions to \eqref{eq:Navier-intro} in the sense of transposition, see Definition \ref{weak solution in the sense of adjoint-isomorphism}), and Theorem \ref{th:constant-indep-T} (for strong solutions, i.e., solutions satisfying \eqref{eq:Navier-intro} a.e. in space and time). These results can be roughly summarized by stating that, given a boundary datum $g$ on $\Gamma_0$, the vector field $P(u)\cdot n$ (along solutions of \eqref{eq:Navier-intro}) is ``one degree less smooth than $g$'' on $\Gamma_0$. These trace regularity results do not follow from the standard regularity of the solutions on $\Omega$, say, from the standard energy estimates and the trace theorem. But, they are all derived through a careful analysis of the equation and a suitable choice of multipliers (see the proofs of Theorems \ref{thm trace regularity with zero B.C.} and \ref{th:strong trace regularity} for a choice of such multipliers). This is inspired by the work \cite{LLT-86} for the scalar wave equation with Dirichlet boundary conditions. For the corresponding Neumann problem, we refer the interested reader to \cite{LT-90}. 

Our results are ``global in time'' existence and regularity results, in the sense that each solution exists and satisfies the stated regularity properties in the same time span $(0,T)$ where the data $F$ and $g$ are defined. A particular feature of our work is that we are able to provide an explicit dependence on $T$ of the constants appearing in our estimates (see equation \eqref{eq:c(T)}). Actually, for strong solutions to \eqref{eq:Navier-intro}, we can make this constant independent of $T$ (see Theorem \ref{th:constant-indep-T}). We believe that this is a remarkable novelty compared to the existing works on the scalar wave equation. It is also a very useful property when dealing with multi-phase problems in which the elastodynamics problem \eqref{eq:Navier-intro} is coupled with other equations, like the Navier--Stokes equations (for fluid-structure interaction problems, see e.g., \cite{kukavica12,kukavica14,raymond2014,kukavica17,kukavica18,boulakia19,kukavica24,mazzone25}). In these types of problems, the coupling arises through continuity conditions on physical quantities (like velocities and stresses) at the fluid-solid interface. Sharp trace regularity results become then a fundamental tool to achieve a proof for the existence of solutions. The majority of the literature just cited either considers a wave-type equation (i.e., \eqref{eq:Navier-intro} with $\lambda+\mu=0$) or  \eqref{eq:Navier-intro} with Dirichlet boundary conditions only, i.e., when a Dirichlet condition is imposed on $\Gamma_1$ too (instead of our Neumann-type one). The corresponding trace results are then used without a proof (with possibly a reference to \cite{LLT-86}). To the best of our knowledge, the results here presented are the first trace regularity results for the Navier equations with mixed Dirichlet- and Neumann-type boundary conditions. Our proofs are novel as they overcome technicalities related to the structure of the equation. As already mentioned, we also exploit the dependence on the maximal existence time $T$ in the constants of our various estimates. We believe that our work could be used as a reference for other problems in continuum mechanics (besides the already cited fluid-structure interaction problems).   

We conclude this introduction with an outline of our paper. In Section \ref{sec:notation}, we introduce some notation on the functional setting used throughout our paper. The main results are stated and proved in Section \ref{sec:results}. Appendix \ref{sec:equalities} provides some useful equalities in multivariable calculus. Finally, in Appendix \ref{sec:equalities2}, we present abstract regularity results for Hilbert space-valued functions. Some of these results can be traced back in \cite{LM-1}. We refine existing results and exploit the time-dependence of some constants. These are novel results of independent interest. 

\section{Notation and basic functional setting}\label{sec:notation}

Throughout the text, we use Einstein summation convention for repeated indices (unless otherwise specified, indices run from 1 to $3$). The symbol $|\cdot|$ is used for both the Euclidean norm of vector fields and the norm associated with the Frobenius inner product of second-order tensors. Also, for a Banach space $X$ and $r>0$, we denote with $\mathbb{B}_{r}(x)$ 
the open ball centered at $X$ with radius $r$. 

Let $\mathcal D$ be a domain in $\R^d$, $d\in \N\setminus\{0\}$, and $p\in [1,\infty]$. We denote by $L^p(\mathcal D)$ and $W^{m,p}(\mathcal D)$ the Lebesgue and Sobolev spaces endowed with the norms 
\[\begin{split}
&\norm{w}_{L^p(\mathcal D)}:=\left(\int_{\mathcal D} |w|^p\;d x\right)^{1/p}\qquad\text{for }p\in[1,\infty),
\\
&\norm{w}_{L^\infty(\mathcal D)}:=\text{ess\;sup}_{\mathcal D}|w(x)|,
\\
&\norm{w}_{W^{m,p}(\mathcal D)}:=\left(\sum_{|\alpha|\le m}\norm{D^\alpha w}^p_{L^p(\mathcal D)}\right)^{1/p},
\end{split}\]
respectively. In the above, $D^\alpha$ denotes the $\alpha$-th order weak derivative of $w$. In the particular case of $p=2$ and vector-valued functions, we recall that the $L^2$-norm is associated to the inner product 
\[
(u,v)_{L^2(\mathcal D)}:=\int_{\mathcal D}u\cdot v\; dx\qquad\qquad\text{for all }u,v\in L^2(\mathcal D). 
\]
We also use the notation $H^m(\mathcal D):=W^{m,2}(\mathcal D)$ for $m\in \N\setminus\{0\}$ and note that $H^m(\mathcal D)$ is a Hilbert space for each $m\in \N\setminus\{0\}$. 

The {\em fractional Sobolev spaces} are defined as follows for $s\in \R$:
\[
H^s(\R^d):=\left\{w\in \mathscr{S}'(\R^d):\; \int_{\R^d}(1+|\xi|^2)^s|\mathscr{F}u(\xi)|^2\; d\xi<\infty\right\},
\]
where $\mathscr{S}'(\R^d)$ is the set of all tempered distributions (i.e., the topological dual of the Schwartz space $\mathscr S$ of rapidly decaying $C^\infty$-functions in $\R^d$), and $\mathscr{F}$ denotes the Fourier transform\footnote{This is the extension by transposition of the classical Fourier transform $\mathscr{F}:\;\mathscr{S}\to \mathscr{S}$ (see \cite[Chapter 1, Section 1, p.45]{LM-1}). }. The space $H^s(\R^d)$ is endowed with the norm
\[
\norm{w}_{H^s(\R^d)}:=\left(\int_{\R^d}(1+|\xi|^2)^s|\mathscr{F}u(\xi)|^2\; d\xi\right)^{1/2}. 
\]
When $s\in \N$, we recover the classical Sobolev spaces, with the latter defining an equivalent norm (\cite[Chapter 1, Section 1, Theorem 1.2]{LM-1}). The above definition of fractional Sobolev spaces carries over domains with boundary provided that suitable extension operators exist (this is the case if $\mathcal D$ is a smooth bounded domain, for example\footnote{See \cite[Chapters 8 \& 11]{leoniBesov} for more general definitions concerning {\em extension domains}.}). In addition, the following characterization via complex interpolation holds: $H^s(\mathcal D)=[H^{s_0}(\mathcal D),H^{s_1}(\mathcal D)]_{\theta}$ with $s_0 < s_1\in \R$, $\theta\in (0,1)$ and $s=(1-\theta)s_0+\theta s_1$ (see \cite[Chapter 1, Section 9]{LM-1}), and there exists a positive constant $k_{\mathcal D}$ such that for all $w\in H^{s_1}(\mathcal D)$:
\[
\norm{w}_{H^s(\mathcal D)}\le k_{\mathcal D}\norm{w}^{1-\theta}_{H^{s_0}(\mathcal D)}\norm{w}^{\theta}_{H^{s_1}(\mathcal D)}. 
\]

As this paper concerns with the trace regularity of functions, we make use of the following fractional Sobolev spaces on the boundary $\partial \mathcal D$. From now on, and for simplicity, we assume that $\mathcal D$ is a bounded domain with smooth boundary and $d>1$. We start by recalling that $L^2(\partial \mathcal D):=L^2(\partial\mathcal D;\sigma)$, where $\sigma:=\mathcal H^{d-1}$ is the $(d-1)$-dimensional Hausdorff measure (see \cite[p. 249]{leoniSobolev}). Then, for $s\in (0,1)$, we define the {\em Besov space}\footnote{The notation used for these spaces is the same as for fractional Sobolev spaces on domains considered earlier. As a matter of fact, similar definitions holds for Besov spaces on domains with norm equivalent to the ones defined above for fractional Sobolev spaces (see \cite{DiNezza}). }
\[
H^s(\partial \mathcal D):=\left\{w\in L^2(\partial\mathcal D):\; \int_{\partial \mathcal D}\int_{\partial\mathcal D}\frac{|u(x)-u(y)|^2}{|x-y|^{d-1+2s}}\;d \sigma(y)d\sigma(x)<\infty\right\}.
\]
This is a Banach space with the norm 
\[
\norm{w}_{H^s(\partial \mathcal D)}:=\left(\norm{w}_{L^2(\partial \mathcal D)}^2+\int_{\partial \mathcal D}\int_{\partial\mathcal D}\frac{|u(x)-u(y)|^2}{|x-y|^{d-1+2s}}\;d \sigma(y)d\sigma(x)\right)^{1/2}. 
\]
It is also an interpolation space between $L^2(\partial\mathcal D)$ and $H^1(\partial\mathcal D)$, specifically it can be characterized through complex interpolation as $H^s(\partial \mathcal D)=[L^2(\partial \mathcal D),H^1(\partial \mathcal D)]_s$ (see \cite[Chapter 1, Section 7, Theorem 7.7]{LM-1} for a more general result). Before we introduce the definition of Besov spaces for $s\ge1$, we recall that since $\partial \mathcal D$ is smooth, for each $x_0\in \partial \mathcal D$ there exists a rigid body motion $T_{x_0}:\; \R^d\to\R^d$ such that $T_{x_0}(x_0)=0$, and there exist a smooth function $f:\;\R^{d-1}\to \R$ with $f(0)=0$, a radius $r_{x_0}>0$, and a diffeomorphism $\Phi:x\in\mathbb{B}_{r_{x_0}}(x_0)\mapsto y\in \mathbb{B}_{r_{x_0}}(0)$ such that   
\begin{equation}\label{eq:parametrization-around-boundary}
y:=\Phi(x)=(T_{x_0}(x)\cdot \mathsf{e}^{1},T_{x_0}(x)\cdot \mathsf{e}^{2},\dots, T_{x_0}(x)\cdot \mathsf{e}^{d-1},T_{x_0}(x)\cdot \mathsf{e}^{d}-f(y_1,\dots,y_{d-1}))
\end{equation}
and 
\begin{equation}\label{eq:parametrization-around-boundary2}\begin{split}
&\Phi(\mathcal D\cap \mathbb{B}_{r_{x_0}}(x_0))=\mathbb{B}_{r_{x_0}}(0)\cap \R^d_+,
\\
&\Phi(\partial\mathcal D\cap\mathbb{B}_{r_{x_0}}(x_0)))=\{y\in \mathbb{B}_{r_{x_0}}(0):\; y_d=0\}=:\mathcal{B}_{r_0}. 
\end{split}\end{equation}
In the above,  $\{\mathsf{e}^{1},\mathsf{e}^{2},\dots,\mathsf{e}^{d}\}$ denotes the canonical basis of $\R^d$, and $\R^d_+:=\{y=(y_1,y_2,\dots,y_{d-1},y_d)\in \R^d:\; y_d>0\}$. We note that $\Phi$ admits a smooth inverse that we denote by $\Theta$. On the other side, since $\mathcal D$ is an open set, for every $x_0\in \mathcal D$, there exists $r_{x_0}>0$ such that $\mathbb{B}_{r_{x_0}}(x_0)\subset \mathcal D$. The union of all the balls $\mathbb{B}_{r_{x_0}}(x_0)$ defined above forms an open cover of $\overline{\mathcal D}$, which is compact. So, there exists a finite number $m$ of points $x^{1}, x^2, \dots, x^m\in \overline{\Omega}$, and corresponding radii $r_1,r_2,\dots,r_m>0$, such that 
\begin{equation}\label{eq:open-cover}
\overline{\mathcal D}\subset \bigcup_{i=1}^m\mathbb{B}_{r_i}(x^i). 
\end{equation}
Now, we consider a partition of unity $\{\psi_i\}_{i=1}^m$ subordinated to the family of balls $\{\mathbb{B}_{r_i}(x^i)\}_{i=1}^m$, and define the {\em Besov space} for $s\in \R$, $s\ge 1$:
\[
H^{s}(\partial\mathcal D):=\{w\in L^2(\partial\mathcal D):\; \norm{w}_{H^{s}(\partial \mathcal D)}<\infty
\} 
\]
where 
\[
\norm{w}_{H^{s}(\partial \mathcal D)}:=\left(\sum^m_{i=1}\norm{(\psi_iw)\circ \Theta|_{\mathcal{B}_{r_i}}}^2_{H^{s}(\R^{d-1})}\right)^{1/2}.
\]
Note that, given the properties of the partition of the unity $\{\psi_i\}^m_{i=1}$, the function $(\psi_iw)\circ \Theta|_{\mathcal{B}_{r_i}}$ is considered to be extended to zero outside the ball $\mathbb{B}_{r_i}(0)$. One can show that the latter definitions are independent of the choice of the reparametrization $\Phi$ as well as of the partition of unity $\{\psi_i\}_i$. These Besov spaces can be also characterized as ``trace spaces'' from the following {\em trace theorem} (which may be used throughout the text without explicitly referring to it). 
\begin{theorem}{\cite[Chapter 1, Section 9, Theorem 9.4.]{LM-1}}
Let $s>1/2$ and denote with $k$ the largest integer satisfying $k<s-1/2$. Let $\mathcal D\subset\R^d$ be a bounded domain with smooth boundary $\Gamma$. Then, the trace operator
\[
\gamma_{s,\Gamma}:\; w\in H^s(\mathcal D)\mapsto \gamma_{s,\Gamma}(w):=\left(w|_\Gamma,\frac{\partial w}{\partial n}|_\Gamma,\dots,\frac{\partial^{k} w}{\partial n^{k}}|_\Gamma\right)\in \prod^{k}_{j=0}H^{s-j-1/2}(\Gamma)
\]
is linear, continuous, surjective with a continuous linear right inverse.
\end{theorem}
For $s>1/2$, we denote by $H^s_0(\mathcal D):=\{w\in H^s(\mathcal D):\; \gamma_s(w)=0\}$. 
With $H^{-1}(\mathcal D)$ we identify the completion of $L^2(\mathcal D)$ with respect to the following norm
\[
\sum_{w\in H^1_0(\mathcal D),\norm{w}_{H^1(\mathcal D)}\le 1}\left|(u,w)_{L^2(\mathcal D)}\right|
=:\norm{u}_{H^{-1}(\mathcal D)}. 
\]
It can be shown that $H^{-1}(\mathcal D)=(H^1_0(\mathcal D))'$ (see \cite[paragraph 3.13]{adams}). 

If $(X, \norm{\cdot}_X)$ is a Banach space and  
$p\in[1,\infty]$, $L^p(0,T;X)$ (resp. $H^k(0,T;X)$, $k\in \N$) denotes the space of strongly measurable functions $f$ from $[0,T]$ to $X$ for which 
\[
\begin{aligned}
&\left(\int_0^T \norm{f(t)}^p_X\; d t\right)^{1/p}<\infty\qquad&\text{for }p\in[1,\infty),
\\
&\text{ess sup}_{t\in(0,T)}\norm{f(t)}^p_X<\infty 
\end{aligned}\]
(resp. $\sum^k_{\ell=0}\left(\int_0^T \norm{\partial^\ell_t f (t)}^2_X\; d t\right)^{1/2}<\infty$). Similarly, $C^k([0,T];X)$ indicates the space of functions which are $k$-times differentiable with values in $X$, and having $\max_{t\in [0,T]}\norm{\partial^\ell_t \cdot}_X < \infty$, for all $\ell = 0,1,...,k$. 

Finally, we note that in our proofs we may use constants $C$, $k$, and/or $c_i$ which change from line to line. After equation \eqref{eq:c(T)}, ``$c(T)$'' denotes a positive constant depending on the time $T$ exactly as in \eqref{eq:c(T)} up to a possibly different multiplicative constant (independent of $T$). 

\section{Main results}\label{sec:results}
In this section, we present the main results of the paper concerning the trace regularity of the stress vector $P(u)\cdot n$ along various types of solution of the following problem (with the same notation introduced in Section \ref{sec:intro}): 
\begin{equation}\label{eq:Navier}
    \begin{aligned}
            &\frac{\partial^{2}u}{\partial t^2} =\mathrm{div} P(u) + F  &&\qquad \text{in } \Omega \times (0,T), \\
            &u(\cdot,0) =u_{0} &&\qquad \text{in } \Omega ,\\
            &\frac{\partial u}{\partial t}(\cdot,0) = u_{1} &&\qquad  \text{in } \Omega, \\
            &u =g && \qquad \text{on } \Gamma_{0} \times (0,T),\\
            & P(u)\cdot n =0 && \qquad \text{on } \Gamma_{1} \times (0,T). 
    \end{aligned}
\end{equation}
Throughout this section, we assume that $\mathcal \sigma(\Gamma_0)>0$. The {\em mixed boundary conditions} \eqref{eq:Navier}$_{4,5}$ require a careful choice of the functional setting for the analysis of the above problem. For $k\in \N\setminus\{0\}$, we introduce the Banach space 
\[
V^k(\Omega):=\{\varphi \in H^k(\Omega):\; \gamma_{k,\Gamma_0}(\varphi)=0\text{ on }\Gamma_0\}
\]
endowed with the norm $\norm{\cdot}_{H^k(\Omega)}$; This is also a Hilbert space with the same inner product of $H^k(\Omega)$. Throughout the text, and for $v\in L^p(0,T;V^k(\Omega))$ (and $p\in[1,\infty]$), we identify $\norm{v}_{L^p(0,T;H^k(\Omega))}$ with the norm $\norm{v}_{L^p(0,T;V^k(\Omega))}$ (a similar note holds for continuous functions with values in $V^k(\Omega)$). 

We notice that $V^k(\Omega)$ is separable and reflexive. In addition, $V^k(\Omega)$ is a dense subset of $L^2(\Omega)$. As a matter of fact, $V^k(\Omega)$ can be also characterized as the completion of the following space 
\[
\{\varphi\in C^\infty(\Omega):\; D^\alpha\varphi|_{\Gamma_0}=0\text{ for each }\alpha\in \N\}
\]
with respect to the $H^k$-norm. Concerning the dual, again for $k\in\N\setminus\{0\}$, we denote by $V^{-k}(\Omega)$ the completion of $L^2(\Omega)$ with respect to the following norm 
\[
\norm{v}_{V^{-k}(\Omega)}:=\sup_{w\in V^k(\Omega),\,\norm{w}_{H^k(\Omega)}\le 1}|(u,w)_{L^2(\Omega)}|.
\]
It can be shown that $V^{-k}(\Omega)=(V^k(\Omega))'$ (see \cite[paragraphs 3.13-3.14]{adams} for the analogous proof in the case of the classical Sobolev spaces). We indicate with $\langle\cdot,\cdot\rangle_{V{-k}(\Omega)}$ the duality pairings between $V^k(\Omega)$ and $V^{-k}(\Omega)$, and note that the following continuous embeddings hold for each $k\in\N\setminus\{0\}$:
\[
H^1_0(\Omega)\hookrightarrow V^k(\Omega)\hookrightarrow H^k(\Omega)\hookrightarrow L^2(\Omega)\hookrightarrow(H^k(\Omega))'\hookrightarrow V^{-k}(\Omega)\hookrightarrow H^{-1}(\Omega). 
\]
We are now ready to present the first important result on the trace regularity of the stress vector $P(u)\cdot n$ along {\em weak solutions} of \eqref{eq:Navier}.

\begin{theorem} \label{thm trace regularity with zero B.C.}
    Take $g=0$ in \eqref{eq:Navier}. If $u_0\in V^1(\Omega)$, 
    $u_1\in L^2(\Omega)$, 
    and $F \in L^{1}(0,T;L^2(\Omega))$, then there exists a unique weak solution $u$ of \eqref{eq:Navier} satisfying the following regularity properties:
    \begin{equation}\label{eq:regularity-distributional}
    u\in L^\infty(0,T;V^1(\Omega)),\quad \frac{\partial u}{\partial t}\in L^\infty(0,T;L^2(\Omega)),
    \quad \frac{\partial^2 u}{\partial t^2}\in L^2(0,T;V^{-1}(\Omega))
    \end{equation}
    In addition, the following trace regularity for the stress vector on $\Gamma_0$ holds: 
    \begin{equation} \label{trace regularity with zero B.C.}
        \norm{P(u)\cdot n}_{L^2(0,T;L^2(\Gamma_0))}\le C\sqrt{T}\left[
        \norm{u_0}_{H^1(\Omega)}
        +\norm{u_1}_{L^2(\Omega)}
        +\norm{F}_{L^1(0,T;L^{2}(\Omega))}\right],
    \end{equation}
    where $C$ is a positive constant independent of $T$. 
\end{theorem}

\begin{remark}\label{rm:distributional-formulation}
    The solution $u$ found in Theorem \ref{thm trace regularity with zero B.C.} satisfies \eqref{eq:Navier} (with $g=0$) in the {\em weak sense}, i.e., 
    \begin{equation}\label{eq:weak-ditributional-formulation}
    \left\langle \frac{\partial^2 u}{\partial t^2},\xi\right\rangle_{V^{-1}(\Omega)}+B(u,\xi)=(F,\xi)_{L^2(\Omega)}
    \end{equation}
    for all $\xi\in V^1(\Omega)$ and a.e. $t\in (0,T)$, where 
    $B:\; V^1(\Omega)\times V^1(\Omega)\to \R$ is a symmetric bilinear form defined by 
    \begin{equation}\label{eq:bilinear}
    B(v,w):=2\mu\int_{\Omega} e(v):e(w)\; dx+\lambda\int_{\Omega}\div v\,\div w\; dx
    \end{equation}
     for all $v,w\in V^1(\Omega)$. By the Poincar\'e-Korn inequality (see \cite[Theorem 6.3-4.]{ciarlet}),  $B$ is a positive definite bilinear form, and it satisfies   
    \begin{equation}\label{eq:B-norm-equivalent}
    k_1\norm{v}^2_{H^1(\Omega)}\le B(v,v)
    \le k_2\norm{v}^2_{H^1(\Omega)}\qquad \text{for all }v\in V,
    \end{equation}
    and some constants $0<k_1<k_2$. 
\end{remark}

\begin{proof}
    The proof of the existence of weak solutions to \eqref{eq:Navier} follows by a classical Faedo-Galerkin method (see \cite{jerin} and \cite[Section 10.1]{valli}).  
    We present here only the apriori (formal energy) estimates. Multiply \eqref{eq:Navier} by $\partial u/\partial t$ and integrate the resulting equation over $\Omega$, using integration by parts, the homogeneous boundary condition ($g=0$), and H\"older's inequality, we find 
    \[
    \frac 12 \frac{d}{d t}\left[ \norm{\frac{\partial u}{\partial t}}^2_{L^2(\Omega)}+B(u,u)\right]
    \le \norm{\frac{\partial u}{\partial t}}_{L^2(\Omega)}\norm{F}_{L^2(\Omega)}.
    \]
    Integrate the latter displayed equation with respect to time and use \eqref{eq:B-norm-equivalent} together with Young's inequality (with $\varepsilon\in (0,1)$), to get 
    \begin{equation}\label{eq:apriori-distributional0bc}\begin{split}
        \norm{\frac{\partial u}{\partial t}(t)}^2_{L^2(\Omega)}&\le \varepsilon \norm{\frac{\partial u}{\partial t}}^2_{L^\infty(0,T;L^2(\Omega))} 
        + \frac{1}{\varepsilon} \norm{F}^2_{L^1(0,T;L^2(\Omega))}
        +\norm{u_1}^2_{L^2(\Omega)}
        +c_1\norm{u_0}^2_{H^1(\Omega)},
        \\
        B(u(t),u(t))&\le 2\norm{\frac{\partial u}{\partial t}}_{L^\infty(0,T;L^2(\Omega))}\norm{F}_{L^1(0,T;L^2(\Omega))}
        +\norm{u_1}^2_{L^2(\Omega)}
        +c_1\norm{u_0}^2_{H^1(\Omega)}.
    \end{split}\end{equation}
    From the first inequality in \eqref{eq:apriori-distributional0bc} and a suitable choice of $\varepsilon$ (for example, $\varepsilon=1/4$), we obtain that 
    \[
    \norm{\frac{\partial u}{\partial t}}_{L^\infty(0,T;L^2(\Omega))}
    \le c_2\left(\norm{F}_{L^1(0,T;L^2(\Omega))}
    +\norm{u_1}_{L^2(\Omega)}
    +\norm{u_0}_{H^1(\Omega)}\right),
    \]
    with a positive constant $c_2$ independent of $T$. Using the latter estimate and \eqref{eq:B-norm-equivalent} in the second inequality of \eqref{eq:apriori-distributional0bc} also yields the following  
    \[
    \norm{u}_{L^\infty(0,T;H^1(\Omega))}\le c_3\left(\norm{F}_{L^1(0,T;L^2(\Omega))}
    +\norm{u_1}_{L^2(\Omega)}
    +\norm{u_0}_{H^1(\Omega)}\right)
    \]
    with a positive constant $c_3$ independent of $T$. As a byproduct of the construction of the solution (and the above estimates), it immediately follows that the solution operator is continuous in the following sense:
    \begin{equation}\label{eq:energy-estimate-distr-0bc}
    \norm{u}_{L^\infty(0,T;H^1(\Omega))}+\norm{\frac{\partial u}{\partial t}}_{L^\infty(0,T;L^2(\Omega))}
    \le k\left(\norm{F}_{L^1(0,T;L^2(\Omega))}
    +\norm{u_1}_{L^2(\Omega)}
    +\norm{u_0}_{H^1(\Omega)}\right)
    \end{equation}
    where $k>0$ is a constant independent of $T$.  
    Uniqueness of solutions holds because the equation is linear. Furthermore, using \cite[Chapter 3, Lemma 8.1]{LM-1}, the above solution can be modified on a set of measure zero so that 
    \begin{equation*}
    u\in C([0,T];H^1(\Omega)),\qquad \frac{\partial u}{\partial t}\in C([0,T];L^2(\Omega))
    \end{equation*}
    and 
    \begin{equation}\label{eq:energy-estimate-distr-0bc-continuous}
    \norm{u}_{C([0,T];H^1(\Omega))}+\norm{\frac{\partial u}{\partial t}}_{C([0,T];L^2(\Omega))}
    \le k\left(\norm{F}_{L^1(0,T;L^2(\Omega))}
    +\norm{u_1}_{L^2(\Omega)}
    +\norm{u_0}_{H^1(\Omega)}\right).
    \end{equation}

    It remains to prove \eqref{trace regularity with zero B.C.}, and this follows from the estimate 
    \begin{equation} \label{inequality with multiplier}
              \norm{P(u) \cdot n}_{L^{2}(0,T;L^2(\Gamma_{0}))} \leq 
              C\sqrt T\left(\norm{F}_{L^1(0,T;L^2(\Omega))}
    +\norm{u_1}_{L^2(\Omega)}
    +\norm{u_0}_{H^1(\Omega)}\right),
    \end{equation}
    with some positive constant $C$ independent of $T$. 
    Thanks to equations \eqref{eq:sym-nabla-u-dot-n-norm} and \eqref{eq:div^2}, to prove the latter 
    it is enough to show that 
      \begin{equation}\label{eq:nabla-div-u-boundary-distrib}
      \norm{\nabla u}_{L^2(0,T;L^2(\Gamma_0))}+\norm{\mathrm{div} u}_{L^2(0,T;L^2(\Gamma_0))}
      \le C\sqrt T \left(\norm{F}_{L^1(0,T;L^2(\Omega))}
    +\norm{u_1}_{L^2(\Omega)}
    +\norm{u_0}_{H^1(\Omega)}\right).
      \end{equation}
      We prove this through formal estimates from \eqref{eq:Navier}, by assuming that $F$ (and the corresponding solution $u$) is smooth, then we conclude by the continuity of the solution operator given by \eqref{eq:energy-estimate-distr-0bc} and the weakly lower semicontinuity of the norm. 
    
    Consider a (time-independent) extension $h\in C^{1}(\overline{\Omega})$ of the outward unit normal $n$, such that 
    \[
    h = n \text{ on $\Gamma_{0}$}, \qquad h = 0 \text{ on $\Gamma_{1}$}. 
    \]
    We next take the $L^2$-inner product of equation \eqref{eq:Navier} with $\nabla u \cdot h$, and we get
      \begin{equation} \label{equation after multiplication in thm1}
          \begin{aligned}
              \int_{0}^{T}\int_{\Omega} \frac{\partial^{2} u }{\partial t^{2}}\cdot(\nabla u \cdot h) \;dx dt&=  \mu\int_{0}^{T}\int_{\Omega} \Delta u \cdot (\nabla u \cdot h) \;dx dt 
              +(\lambda+\mu)\int_{0}^{T}\int_{\Omega} \nabla(\mathrm{div}u)\cdot (\nabla u \cdot h) \;dx dt 
              \\ 
              &\quad + \int_{0}^{T}\int_{\Omega} F \cdot(\nabla u \cdot h) \;dx dt,
          \end{aligned}
      \end{equation}
      here we have used the fact $\mathrm{div} P(u)=\mu \Delta u +(\lambda+\mu) \nabla(\mathrm{div}u)$.
      
      Let us examine each term of the above equation separately. Throughout next calculations, we use Fubini's Theorem, integration by parts and the properties of $h$ on the boundary. The term on the left-hand side of \eqref{equation after multiplication in thm1} can be rewritten as follows: 
     \begin{equation} \label{eq:lhs}
        \begin{aligned}
            \int_{0}^{T}\int_{\Omega} \frac{\partial^{2} u }{\partial t^{2}}\cdot (\nabla u \cdot h) &\;dx dt
            \\
            &=\int_\Omega \frac{\partial u}{\partial t}(\cdot,T)\cdot(\nabla u(\cdot,T)\cdot h)\; dx
            -\int_{\Omega}u_{1} \cdot (\nabla u_{0} \cdot h) dx
            + \frac 12 \int_{0}^{T}\int_{\Omega}\left(\frac{\partial u}{\partial t}\right)^{2} \mathrm{div} h \;dx dt.
         \end{aligned}
     \end{equation} 
    Let us consider the first term on the right-hand side of \eqref{equation after multiplication in thm1}. Proceeding similarly as before and using equation \eqref{eq:nabla-u-dot-n-norm}, we find that
     \begin{equation} \label{eq:rhs_1st_part}
         \begin{aligned}
         \int_{\Omega}\Delta u \cdot (\nabla u \cdot h)\; dx &
         =\int_{\Gamma_{0}}|\nabla u\cdot n|^2\;d\sigma
             -\frac 12\int_{\Gamma_{0}}|\nabla u|^2\;d\sigma
             +\frac12\int_{\Omega}|\nabla u|^2\div{h}\; dx
             -\int_{\Omega}\nabla u : (\nabla u\cdot\nabla h)\; dx
        \\
        &=\frac 12\int_{\Gamma_{0}}|\nabla u|^2\; d\sigma
             +\frac12\int_{\Omega}|\nabla u|^2\div{h}\; dx
             -\int_{\Omega}\nabla u : (\nabla u\cdot\nabla h)\; dx.
         \end{aligned}
     \end{equation}
     Now, for the second term on the right-hand side of \eqref{equation after multiplication in thm1}, using \eqref{eq:div^2} we obtain that 
     \begin{equation}\label{rhs_2nd term}
     \begin{aligned}
         &\int_{\Omega} \nabla(\mathrm{div}u)\cdot (\nabla u \cdot h) \;dx 
         \\
         &=\int_{\Gamma_{0}}\div {u} (n\cdot \nabla u \cdot n)\; d\sigma
         -\frac{1}{2}\int_{\Gamma_{0}}(\div{u})^2\;d\sigma
         +\frac 12 \int_{\Omega}(\div{u})^2\div{h}\; dx
         -\int_{\Omega}\div{u}\nabla u:(\nabla h)^T\; dx
         \\
         &=\frac 12 \int_{\Gamma_{0}}(\div{u})^2\;d\sigma
         +\frac 12 \int_{\Omega}(\div{u})^2\div{h}\; dx
         -\int_{\Omega}\div{u}\nabla u:(\nabla h)^T\; dx.
     \end{aligned}      
     \end{equation}
    From \eqref{eq:lhs}, \eqref{eq:rhs_1st_part} and \eqref{rhs_2nd term} we finally have that 
     \begin{equation} \label{eq:final equality of thm 3.1}
         \begin{aligned}
            &\frac 12 \int_{0}^{T}\int_{\Gamma_{0}}\left[\mu |\nabla u|^{2}+(\lambda+\mu)(\div{u})^{2}\right]\;d\sigma dt
            \\
            &\quad=\int_\Omega \frac{\partial u}{\partial t}(\cdot,T)\cdot(\nabla u(\cdot,T)\cdot h)\; dx
            -\int_{\Omega}u_{1} \cdot (\nabla u_{0} \cdot h) dx
            + \frac 12 \int_{0}^{T}\int_{\Omega}\left(\frac{\partial u}{\partial t}\right)^{2} \mathrm{div} h \;dx dt
            \\
            &\qquad -\frac\mu2\int^T_0\int_{\Omega}|\nabla u|^2\div{h}\; dxdt
            +\mu \int^T_0\int_{\Omega}\nabla u : (\nabla u\cdot\nabla h)\; dxdt
            \\
            &\qquad -\frac{\lambda+\mu}{2} \int^T_0\int_{\Omega}(\div{u})^2\div{h}\; dxdt
            +(\lambda+\mu)\int^T_0\int_{\Omega}\div{u}\nabla u:(\nabla h)^T\; dxdt
            \\
            &\qquad -\int_{0}^{T}\int_{\Omega} F \cdot(\nabla u \cdot h) \;dx dt. 
        \end{aligned}
     \end{equation}
     
     Each term on the right-hand side of \eqref{eq:final equality of thm 3.1} can be estimated using H\"older's inequality and \eqref{eq:energy-estimate-distr-0bc-continuous}, thus yielding \eqref{eq:nabla-div-u-boundary-distrib}. 
\end{proof}

\begin{remark}\label{rm:continuity-distributional-sln}
As already mentioned in the above proof, by \cite[Chapter 3, Lemma 8.1]{LM-1}, the weak solution $u$ can be modified on a set of measure zero so that 
\begin{equation}\label{eq:continuity-distributional-sln}
u\in C([0,T];V^1(\Omega)),\qquad \frac{\partial u}{\partial t}\in C([0,T];L^2(\Omega)).
\end{equation}
In particular, $u(\cdot,0)=u_0$ and $\displaystyle\frac{\partial u}{\partial t}(\cdot,0)=u_1$, and the following estimate holds with a positive constant $C$ independent of $T$:  
\begin{multline}\label{eq:sharp-energy-estimate-distr-0bc-continuous}
    \norm{u}_{C([0,T];H^1(\Omega))}+\norm{\frac{\partial u}{\partial t}}_{C([0,T];L^2(\Omega))}+\norm{P(u)\cdot n}_{L^2(0,T;L^2(\Gamma_0))}
    \\
    \le C(1+\sqrt T) \left(\norm{F}_{L^1(0,T;L^2(\Omega))}
    +\norm{u_1}_{L^2(\Omega)}
    +\norm{u_0}_{H^1(\Omega)}\right).
    \end{multline}
Furthermore, we have that 
\begin{equation}\label{eq:divPu-weak}
\div P(u)=\frac{\partial^2 u}{\partial t^2}-F\in L^2(0,T;V^{-1}(\Omega)).
\end{equation}
\end{remark}

Next, we consider even ``weaker'' solutions of \eqref{eq:Navier}, in the sense of the definition below. 
By Theorem \ref{thm trace regularity with zero B.C.}, we know that for any 
$\psi\in L^1(0,T;L^2(\Omega))$ 
there exists a unique weak solution of \eqref{eq:Navier} with $g=u_0=u_1=0$. In turn, this implies that the following problem admits a unique weak solution 
        \begin{equation} \label{eq with change of variable}
            \begin{aligned}
            \frac{\partial^{2}\phi}{\partial t^2} &=\mathrm{div} P(\phi) + \psi 
                     && \quad \text{in }\Omega \times (0,T), 
            \\
           \phi(\cdot,T) &=0; &&\quad\text{in }\Omega,
           \\
           \frac{\partial \phi}{\partial t}(\cdot,T) &= 0 &&\quad  \text{in }\Omega,
           \\
           P(\phi)\cdot n &=0 &&\quad  \text{on }\Gamma_{1}\times (0,T), 
           \\ 
           \phi &=0 &&\quad  \text{on } \Gamma_{0} \times (0,T). 
           \end{aligned}
        \end{equation}
In view of Remarks \ref{rm:distributional-formulation} and \ref{rm:continuity-distributional-sln}, we have that $\phi$ satisfies 
\begin{equation}\label{eq:IBVP-final-time}
\begin{aligned}
&\left\langle \frac{\partial^2 \phi}{\partial t^2},\xi\right\rangle_{V^{-1}(\Omega)}+B(\phi,\xi)=(\psi,\xi)_{L^2(\Omega)}&&\text{for all $\xi\in  V^1(\Omega)$ and a.e. $t\in (0,T)$,}
\\
&\phi(\cdot,T) =\frac{\partial \phi}{\partial t}(\cdot,T) = 0 &&\text{in }\Omega,
\end{aligned}
\end{equation}
where the bilinear form $B$ has been defined in \eqref{eq:bilinear}. In addition, the solution $\phi$ enjoys the estimates \eqref{eq:energy-estimate-distr-0bc} and \eqref{trace regularity with zero B.C.} with $u=\phi$, $F=\psi$ and zero initial data. In turn, the map
\begin{equation}\label{eq:isomorphism-S}
S:\; \phi \in L^{2}(0,T;V^{1}(\Omega)) \cap H^{1}(0,T;L^{2}(\Omega)) \cap H^{2}(0,T;V^{-1}(\Omega)) 
\mapsto \psi\in L^1(0,T;L^2(\Omega))
\end{equation}
is an isomorphism with a bounded inverse. 

Assume that $\psi$ is smooth in \eqref{eq with change of variable}. Consider \eqref{eq:Navier} with all the data being smooth. Now multiply equation \eqref{eq:Navier} by $\phi$ and equation \eqref{eq with change of variable} by $u$, after integrating by parts, we get
\begin{equation} \label{definition of weak soln}
\begin{split}
    \int_{0}^{T} (u,\psi)_{L^{2}(\Omega)} dt&=\langle u_1,\phi(\cdot,0)\rangle_{V^{-1}(\Omega)}
    -\left(u_0,\frac{\partial \phi}{\partial t}(\cdot,0)\right)_{L^2(\Omega)}
    \\
    &\quad
    -\int_{0}^{T} \int_{\Gamma_{0}}g \cdot (P(\phi)\cdot n) d\sigma dt
    +\int^T_0\langle F,\phi\rangle_{V^{-1}(\Omega)}\; dt.
\end{split}\end{equation} 
The latter equation justifies the following definition. 
\begin{definition}[Solutions in the sense of transposition] 
\label{weak solution in the sense of adjoint-isomorphism}
We say that $u\in L^\infty(0,T;L^2(\Omega))$ 
is a {\em solution to \eqref{eq:Navier} in the sense of transposition} if \eqref{definition of weak soln} is satisfied for all $\psi\in L^1(0,T;L^2(\Omega))$,
with $\phi=S^{-1}\psi$ and $S$ defined in \eqref{eq:isomorphism-S}. 
\end{definition}
In the next lemma, we show that the class of solutions in the sense of transposition is non-empty. 
\begin{lemma} \label{existence of weak solution in the sense of adjoint-isomorphism}
If $F\in L^{1}(0,T;V^{-1}(\Omega))$, $u_{0}\in L^{2} (\Omega)$, $u_{1} \in V^{-1}(\Omega)$, and  $g \in L^{2}(0,T;L^2(\Gamma_{0}))$, then there exists a unique solution in the sense of transposition to \eqref{eq:Navier} satisfying 
    \begin{equation} \label{eq: regularity of very weak solution}
      \begin{split}
      u &\in L^{\infty}(0,T;L^{2}(\Omega)),
      \\
      \frac{\partial u}{\partial t} &\in L^{\infty}(0,T;V^{-1}(\Omega)). 
      \end{split}  
    \end{equation}  
In addition, there exists a positive constant $c(T)$, depending on $T$ (see \eqref{eq:c(T)} below), such that 
\begin{equation}\label{eq:energy-estimate-weak-adjoint}
\begin{split}
\norm{u}_{L^\infty(0,T;L^2(\Omega))}&+\norm{\frac{\partial u}{\partial t}}_{L^\infty(0,T;V^{-1}(\Omega))}
\\
&\quad\le  
c(T)\left(\norm{u_1}_{V^{-1}(\Omega)}+\norm{u_0}_{L^{2}(\Omega)}+\norm{g}_{L^{2}(0,T;L^2(\Gamma_{0}))}+\norm{F}_{L^1(0,T;V^{-1}(\Omega))}\right).
\end{split}
\end{equation}
Furthermore, after modification on a set of measure zero, 
\begin{equation} \label{eq:continuous regularity of very weak solution}
u \in C([0,T];L^2(\Omega)), \qquad \frac{\partial u}{\partial t} \in C([0,T];V^{-1}(\Omega)), 
\end{equation}  
and the above estimate holds with the $L^\infty(0,T;\cdot)$-norm replaced by the $C([0,T];\cdot)$-norm.
\end{lemma}

\begin{proof}
Let $ \psi\in L^1(0,T;L^2(\Omega))$, and consider $\phi=S^{-1}\psi$, where $S$ is defined in \eqref{eq:isomorphism-S}. Specifically, $\phi$ is a weak solution  of \eqref{eq with change of variable}. So, Theorem \ref{thm trace regularity with zero B.C.} and \eqref{eq:sharp-energy-estimate-distr-0bc-continuous} applied to $\phi$ imply that the map  
\begin{equation} \label{eq:mapping from faedo galerkin}
\begin{split}
    \psi\in L^1(0,T;L^2(\Omega))
    \mapsto &\;\langle u_1,\phi(\cdot,0)\rangle_{V^{-1}(\Omega)}
    -\left(u_0,\left.\frac{\partial \phi}{\partial t}\right|_{t=0}\right)_{L^2(\Omega)}
    \\
    &\quad-\int_{0}^{T} \int_{\Gamma_{0}}g \cdot (P(\phi)\cdot n) d\sigma dt
    +\int^T_0\langle F,\phi\rangle_{V^{-1}(\Omega)}\; dt \in \R
\end{split}\end{equation} 
is linear and continuous. By Riesz representation theorem (\cite[Theorem 8.17]{leoniSobolev}), there exists $u\in L^\infty(0,T;L^2(\Omega))$ such that \eqref{definition of weak soln} holds. Uniqueness follows immediately from \eqref{definition of weak soln} by taking $\psi\in C^\infty_c(\Omega\times[0,T))$. 

It remains to show that $u$ satisfies \eqref{eq: regularity of very weak solution}$_2$. Consider the map
\[
\xi\in L^1(0,T;V^2(\Omega))
\mapsto 
\int^T_0\int_{\Omega}u\cdot\div P(\xi)\; dx dt+\int^T_0<F,\xi>_{V^{-1}(\Omega)}\; dt\;\in \R;
\]
this is a bounded linear functional. So, again by the Riesz representation theorem (\cite[Theorem 8.17]{leoniSobolev}), there exists $z\in L^\infty(0,T;V^{-2}(\Omega))$ such that 
\begin{equation}\label{eq:d2dt-u}
\int^T_0\int_{\Omega}u\cdot\div P(\xi)\; dx dt+\int^T_0<F,\xi>_{V^{-1}(\Omega)}\; dt
=\int^T_0 <z,\xi>_{V^{-2}(\Omega)}. 
\end{equation}
Now, take $\xi(x,t)=\varphi(t)\chi(x)$ with $\varphi\in C^\infty_c(0,T;\R)$ and $\chi\in V^2(\Omega;\R^3)$. 
From \eqref{eq:divPu-weak} and \eqref{definition of weak soln},\footnote{\label{foot:regularity-S}The operator $S$ defined in \eqref{eq:isomorphism-S} is still a homeomorphism from 
$L^2(0,T;H^{k+1}(\Omega)\cap V^1(\Omega))\cap H^1(0,T;H^k(\Omega)\cap V^1(\Omega))\cap H^2(0,T;H^{k-1}(\Omega))$ to $L^1(0,T;H^{k}(\Omega))$ for $k\ge 1$,
and from $C^\infty_c(\Omega\times(0,T))$ into $C^\infty_c(\Omega\times(0,T))$
(this is basically due to the fact that the PDE is linear). } it follows that 
\[
\int^T_0 <z,\varphi\chi>_{V^{-2}(\Omega)}=\int^T_0\left( u,\varphi''\chi\right)_{L^2(\Omega)}dt.
\]
Using \cite[Remarks 8.8 \& 8.50]{leoniSobolev}, we therefore conclude that $\displaystyle\frac{\partial^2 u}{\partial t^2}=z\in L^\infty(0,T;V^{-2}(\Omega))$, and we obtain the second condition in \eqref{eq: regularity of very weak solution} by Corollary \ref{cor:continuity-time-Linfty-Hm}. 

Finally, we notice that \eqref{eq:energy-estimate-weak-adjoint} follows from \eqref{eq:mapping from faedo galerkin}, \eqref{eq:d2dt-u},  and \eqref{eq:continuity-time-Linfty-Hm}. In particular, 
\begin{equation}\label{eq:c(T)}
c(T):=CT^{-1}\max\{1,T^2\}(1+T), 
\end{equation} 
with $C$ a positive constant independent of $T$. Furthermore, \eqref{eq:continuous regularity of very weak solution} can be derived using \cite[Chapter 3, Lemma 8.1]{LM-1}.  
\end{proof}

Next, we use the latter lemma to prove the existence of weak solutions to \eqref{eq:Navier} corresponding to nonhomogeneous boundary conditions.  
\begin{lemma} \label{Lemma: existence of weak distribution} 
    In \eqref{eq:Navier}, we now take $F=0$, 
      \begin{equation*} \label{lemma 2 assumptions}
         g,\ \frac{\partial g}{\partial t} \in L^{2}(0,T;L^2(\Gamma_{0})), \quad  u_{0}\in H^{1}(\Omega) \quad \text{and} \quad u_{1}\in L^{2}(\Omega),
      \end{equation*}
      and we suppose that the following {\em compatibility condition} holds:
        \begin{equation} \label{compatibility condition of lemma 2}
           g(\cdot,0)=u_{0}|_{\Gamma_{0}}.
        \end{equation}
        Then there exists a solution $u$ (in the sense of transposition) of \eqref{eq:Navier} satisfying:
        \begin{equation} \label{regularity lemma 2}
           \begin{aligned}
               u, \frac{\partial u}{\partial t} \in C([0,T];L^{2}(\Omega)), \qquad
               \frac{\partial^{2}u}{\partial t^{2}} \in C([0,T];V^{-1}(\Omega)) . 
           \end{aligned} 
        \end{equation}
        Additionally, if 
        \begin{equation} \label{additional assumptions on lemma 2}
            g\in L^{\infty}(0,T;H^{\frac{1}{2}}(\Gamma_{0}));
        \end{equation}
        then the above solution $u$ further satisfies:
        \begin{equation} \label{additional regularity of lemma 2} 
            u \in C([0,T];H^{1}(\Omega))
        \end{equation}
        and 
        \begin{multline}\label{eq:energy-estimate-adjoint-regular}
        \norm{u}_{C([0,T];H^1(\Omega))}+\norm{\frac{\partial u}{\partial t}}_{C([0,T];L^2(\Omega))}
        \\ \le c(T)
        \left(\norm{u_0}_{H^1(\Omega)}+\norm{u_1}_{L^2(\Omega)}+\norm{g}_{L^\infty(0,T;H^{1/2}(\Gamma_0))}
        +\norm{\frac{\partial g}{\partial t}}_{L^2(0,T;L^2(\Gamma_0))}\right)
        \end{multline}
        with  $c(T)$ a positive constant (depending on $T$). 
\end{lemma}

\begin{proof}
    From Lemma \ref{existence of weak solution in the sense of adjoint-isomorphism}, corresponding to the data $F=0$, $g$, $u_0$, and $u_1$, there exists a unique solution $u$ (in the sense of transposition) of \eqref{eq:Navier} satisfying 
    \eqref{eq:energy-estimate-weak-adjoint} and \eqref{eq:continuous regularity of very weak solution}.  Assume now that all the data are smooth (then use the fact that the solution operator -- defined as a byproduct of Lemma \ref{existence of weak solution in the sense of adjoint-isomorphism} -- has a continuous inverse in the topology given by \eqref{eq:energy-estimate-weak-adjoint}). Then, again by Lemma \ref{existence of weak solution in the sense of adjoint-isomorphism}, 
    corresponding to the data $\tilde F=0$, $\tilde g=\partial g/\partial t$, $\tilde u_0=u_1$, and $\tilde u_1=\div P(u_0)$, there exists a unique solution $\tilde u$ (in the sense of transposition) of \eqref{eq:Navier} enjoying  
    \eqref{eq:energy-estimate-weak-adjoint} and \eqref{eq:continuous regularity of very weak solution} with $u$ replaced by $\tilde u$. In particular, from \eqref{definition of weak soln}, for all $\psi\in C^\infty_c(\Omega\times(0,T))$ (and correspondingly $\phi=S^{-1}\psi \in C^\infty_c(\Omega\times(0,T))$, see Footnote \ref{foot:regularity-S}) we have that 
    \begin{equation}
    \begin{split}
    \int_{0}^{T} \left(u,\frac{\partial \psi}{\partial t}\right)_{L^{2}(\Omega)} dt
    =- \int_{0}^{T} (\tilde u,\psi)_{L^{2}(\Omega)} dt. 
    \end{split}\end{equation}
    Thus, \[
    \frac{\partial u}{\partial t}=\tilde u\in C([0,T];L^2(\Omega))
    \]
    and in a similar fashion, 
    \[
    \frac{\partial^2 u}{\partial t^2}=\frac{\partial \tilde u}{\partial t}\in C([0,T];V^{-1}(\Omega)).
    \]
    Consider the unique solution $\hat u\in L^\infty(0,T;H^1(\Omega))$ of the elliptic problem\footnote{The existence of such a solution is an immediate consequence of the Lax-Milgram Theorem applied to the bilinear form defined in \eqref{eq:bilinear}, and thanks to \eqref{eq:B-norm-equivalent}. } 
    \[\left\{\begin{aligned}
        &\div P(\hat u)=\frac{\partial^2 u}{\partial t^2}-\div P(\hat g)\quad&&\text{in }\Omega,\text{ and a.e. }t\in (0,T), 
        \\
        &\hat u=0\quad &&\text{on }\Gamma_0,\text{ and a.e. }t\in (0,T), 
        \\
        &P(\hat u)\cdot n=0 &&\text{on }\Gamma_1,\text{ and a.e. }t\in (0,T), 
    \end{aligned}\right.\]
    where $\hat g\in L^\infty(0,T;H^1(\Omega))$ is a lift of $g\in L^\infty(0,T;H^{1/2}(\Gamma_0))$ on $\Omega$, and such that $\hat g=0$ on $\Gamma_1$ for a.e. $t\in (0,T)$. By \eqref{definition of weak soln}, it follows that 
    \[
    u=\hat u+\hat g\in L^\infty(0,T;H^1(\Omega)),
    \]
    and 
    \begin{equation}\label{eq:elliptic-extension-H1}
    \norm{u}_{L^\infty(0,T;H^1(\Omega))}\le c_1\left(\norm{\frac{\partial \tilde u}{\partial t}}_{L^\infty(0,T;V^{-1}(\Omega))}
    +\norm{g}_{L^\infty(0,T;H^{1/2}(\Gamma_0))}\right). 
    \end{equation}
    
    Finally, estimate \eqref{eq:energy-estimate-adjoint-regular} follows from \eqref{eq:energy-estimate-weak-adjoint} applied to  $u$ and $\tilde u$, and from the estimate \eqref{eq:elliptic-extension-H1} (after a possible modification of the solution on a set of measure zero, by \cite[Chapter 3, Lemma 8.1]{LM-1}). 
    \end{proof}
    
    In Lemma \ref{Lemma: existence of weak distribution} we have proved the existence of solutions corresponding to nonhomegeneous boundary conditions satisfying $g$ and $\displaystyle\frac{\partial g}{\partial t} \in L^{2}(0,T;L^2(\Gamma_{0}))$, and $g\in L^{\infty}(0,T;H^\frac{1}{2}(\Gamma_{0}))$. Next, we extend this result to the case of a nonzero force $F$ and a slightly more regular $g$.
\begin{theorem} \label{thm trace regularity with nonzero B.C.}
   Let $F \in L^{1}(0,T;L^{2}(\Omega))$, $u_{0} \in H^{1}(\Omega)$, $u_{1} \in L^{2}(\Omega)$, and 
    \[
    g\in H^1(0,T;L^2(\Gamma_0))\cap L^2(0,T;H^1(\Gamma_0))  
    \]
    satisfying the compatibility condition $g(\cdot,0)=u_{0}$. Then, there exists a unique solution $u$ to \eqref{eq:Navier} with 
   \begin{equation} \label{eq:strong solution of u}
       u\in C([0,T];H^{1}(\Omega)), \qquad \frac{\partial u}{\partial t}\in C([0,T];L^{2}(\Omega)), \qquad \frac{\partial^{2}u}{\partial t^{2}} \in C([0,T];H^{-1}(\Omega)),
   \end{equation}
   and enjoying the following estimate: 
   \begin{equation} \label{trace regularity for nonhomogeneous case}
   \begin{aligned}
        \norm{u}_{C([0,T];H^1(\Omega))}
        &+\norm{\frac{\partial u}{\partial t}}_{C([0,T];L^2(\Omega))}
        +\norm{P(u)\cdot n}_{L^2(0,T;L^2(\Gamma_0))}
        \\
        &\le c(T)
        \left(\norm{F}_{L^{1}(0,T;L^{2}(\Omega))}+\norm{u_{0}}_{H^{1}(\Omega)}
        +\norm{u_{1}}_{L^{2}(\Omega))}+\norm{g}_{H^{1}(\Gamma_{0} \times (0,T))}\right),
   \end{aligned}
   \end{equation}
   with $c(T)$ is a positive (constant depending on $T$). 
\end{theorem}
\begin{proof}
    By Theorem \ref{thm trace regularity with zero B.C.} and the superposition principle, it suffices to prove the result for $F=0$. According to the assumptions $g \in L^{2}(0,T;H^{1}(\Gamma_{0})) \cap H^{1}(0,T;L^{2}(\Gamma_{0}))$, so that $g$ satisfies \eqref{additional assumptions on lemma 2} (by Theorem \ref{th:continuity-time-L2-Hm} with $m=1$,  $V=H^{1}(\Gamma_{0})$, and $H=L^{2}(\Gamma_{0})$). Therefore, \eqref{eq:strong solution of u} follows from \eqref{additional regularity of lemma 2} and \eqref{regularity lemma 2}. In addition, by \eqref{eq:energy-estimate-adjoint-regular}, we have  
    \begin{equation} \label{eq: inequality for nonzero b.c. in thm 2 }
    \norm{\nabla u}_{C([0,T];L^{2}(\Omega))}+\norm{\frac{\partial u}{\partial t}}_{C([0,T];L^{2}(\Omega))} \leq c(T)E
    \end{equation}
    where we have set
    \begin{equation} \label{eq:norm with b.c. and i.c.}
    \begin{aligned}
    E:=\norm{F}_{L^{1}(0,T;L^{2}(\Omega))}+\norm{u_{0}}_{H^{1}(\Omega)}+\norm{u_{1}}_{L^{2}(\Omega))}+\norm{g}_{H^{1}(0,T;L^2(\Gamma_{0}))}+\norm{g}_{L^{2}(0,T;H^1(\Gamma_{0}))}. 
    \end{aligned}
    \end{equation}
    
    To prove \eqref{trace regularity for nonhomogeneous case} we only need to show
     \begin{equation} \label{eq:trace regularity with nonzero b.c.}
    \norm{P(u) \cdot n}_{L^{2}(0,T;L^2(\Gamma_{0}))} \leq c(T)E. 
    \end{equation}
    The latter is true if, in particular, the following estimate holds:
    \begin{equation}\label{eq:estimate-gradu-dot-n-gradu-divu}
    \sqrt3 \mu\norm{\nabla u\cdot n}_{L^2(0,T;L^2(\Gamma_0))}+\mu\norm{\nabla u}_{L^2(0,T;L^2(\Gamma_0))}+\lambda\norm{\div u}_{L^2(0,T;L^2(\Gamma_0))}\leq c(T)E. 
    \end{equation}
    We further notice that the subsequent equalities hold on $\Gamma_0$, if $u$ is sufficiently smooth:
    \[
    \nabla u=\nabla u\cdot n\otimes n+\nabla \hat g-\nabla \hat g\cdot n \otimes n
    \]
    and 
    \[\begin{split}
    &|\nabla u|^2=|\nabla u\cdot n|^2+|\nabla \hat g-\nabla \hat g\cdot n \otimes n|^2,
    \end{split}\]
    where $\hat g\in L^2(0,T;H^{3/2}(\Omega))\cap H^1(0,T;H^{1/2}(\Omega))$ is the solution (with fixed time) of the following elliptic problem 
    \begin{equation}\label{eq:divP0-extension}
    \left\{\begin{aligned}
    &\Delta w=0\qquad&&\text{in }\Omega,
    \\
    &w=g\qquad&&\text{on }\Gamma_0,
    \\
    &w=0\qquad&&\text{on }\Gamma_1.
    \end{aligned}\right. 
    \end{equation}
    The existence of such an extension $\hat g$ of the boundary datum $g$ is classical, we refer the interested reader to \cite[Chapter 2, Section 7, Theorems 7.3.]{LM-1}. In particular, from \cite[Chapter 2, Section 7, equation (7.28)]{LM-1}, we have that the estimate 
    \begin{equation}\label{eq:estimate-extension-g}
    \norm{\hat g}_{L^2(0,T;H^{3/2}(\Omega))}+\norm{\frac{\partial \hat g}{\partial t}}_{L^2(0,T;H^{1/2}(\Omega))}\le c_1
    \left(\norm{g}_{L^2(0,T;H^1(\Gamma_0))}+\norm{\frac{\partial g}{\partial t}}_{L^2(0,T;L^2(\Gamma_0))}\right)
    \end{equation}
    with $c_1$ a positive constant independent of $T$. Thus, 
    \begin{equation}\label{eq:estimate-nablau-with-nablaudotn-g}
    \norm{\nabla u}_{L^2(0,T;L^2(\Gamma_0))}^2\le \norm{\nabla u\cdot n}_{L^2(0,T;L^2(\Gamma_0))}^2
    +4c_1\left(\norm{g}_{L^2(0,T;H^1(\Gamma_0))}^2+\norm{\frac{\partial g}{\partial t}}_{L^2(0,T;L^2(\Gamma_0))}^2\right),\end{equation}
    and since $\norm{\div u}_{L^2(0,T;L^2(\Gamma_0))}^2\le 2\norm{\nabla u}_{L^2(0,T;L^2(\Gamma_0))}^2$, \eqref{eq:estimate-gradu-dot-n-gradu-divu} holds if 
    \begin{equation}\label{eq:estimate-gradu-dot-n}
    \norm{\nabla u\cdot n}_{L^2(0,T;L^2(\Gamma_0))}\leq c(T)E. 
    \end{equation}
    Let us then focus to prove the latter estimate. Assuming that all the data are smooth\footnote{We note that the {\em solution operator} that maps the given data to the (unique) solution of \eqref{eq:Navier} (in the sense of transposition) is continuous in the topologies given by \eqref{eq:energy-estimate-weak-adjoint} and \eqref{eq: inequality for nonzero b.c. in thm 2 }-\eqref{eq:norm with b.c. and i.c.}, respectively} and following the same strategy as in the proof of Theorem \ref{thm trace regularity with zero B.C.}, we multiply \eqref{eq:Navier} by $\nabla u \cdot h$. We then handle the extra term appearing in \eqref{eq:lhs}, due to the non-zero boundary data $g$, as follows 
    \begin{equation} \label{eq:lhs of thm 2}
         \begin{aligned}
            \int_{\Omega}\int_{0}^{T} \frac{\partial^{2} u }{\partial t^{2}}  (\nabla u \cdot h)\; dt dx 
            &=\int_{\Omega}\left(\frac{\partial u}{\partial t}\bigg|_{t=T}\right) \cdot (\nabla u (T) \cdot h) \; dx
            -\int_{\Omega} u_{1} \cdot (\nabla u_{0} \cdot h) \; dx
            \\
            &\quad +\int_{0}^{T}\int_{\Omega}\frac{1}{2} \left(\frac{\partial u}{\partial t}\right)^{2} \cdot \mathrm{div} h\; dx dt 
            -\int_{0}^{T}\int_{\Gamma_{0}}\frac{1}{2} \left(\frac{\partial g}{\partial t}\right) ^{2}\; d\sigma dt \\
            & \quad \leq  c(T)E. 
         \end{aligned}
     \end{equation}
     For the remaining terms, we use only the first equalities in \eqref{eq:rhs_1st_part} and \eqref{rhs_2nd term}, and estimate the volume integrals using \eqref{eq: inequality for nonzero b.c. in thm 2 }. Combining these estimates with \eqref{eq:lhs of thm 2}, we obtain the following   
     \[\begin{split}
        \mu &\left[\int^T_0\int_{\Gamma_{0}}|\nabla u\cdot n|^2\;d\sigma dt
             -\frac 12\int^T_0\int_{\Gamma_{0}}|\nabla u|^2\;d\sigma dt\right] 
             \\
        &\quad +(\mu+\lambda)\left[\int^T_0\int_{\Gamma_{0}}\div {u} (n\cdot \nabla u \cdot n)\; d\sigma dt
         -\frac{1}{2}\int^T_0\int_{\Gamma_{0}}(\div{u})^2\;d\sigma dt\right]\le c(T)E.
         \end{split}\]
         On the first two terms on the left-hand side of the latter displayed equation, we use \eqref{eq:estimate-nablau-with-nablaudotn-g}. For the remaining two integrals, we notice that 
         \[
         \div u(n\cdot \nabla u\cdot n)-\frac 12(\div u)^2=\frac 12 (\div u)^2 
         -\div u(\div \hat g-n\cdot \nabla \hat g\cdot n)\ge -\frac 12 (\div \hat g-n\cdot \nabla \hat g\cdot n)^2.
         \]
         Therefore, we obtain 
         \begin{multline*}
         \frac \mu2\norm{\nabla u\cdot n}_{L^2(0,T;L^2(\Gamma_0))}^2\le c(T)E
         +2c_1\mu\left(\norm{g}_{L^2(0,T;H^1(\Gamma_0))}^2+\norm{\frac{\partial g}{\partial t}}_{L^2(0,T;L^2(\Gamma_0))}^2\right)
         \\
         +\frac{\mu+\lambda}{2} \norm{\div \hat g-n\cdot \nabla \hat g\cdot n}_{L^2(0,T;L^2(\Gamma_0))}^2,
         \end{multline*}
         and \eqref{eq:estimate-gradu-dot-n} immediately follows from the last estimate, \eqref{eq:estimate-extension-g}, and \eqref{eq:norm with b.c. and i.c.}. 
\end{proof}

Using the latter theorem, we are ready to demonstrate the following result on the trace regularity of the stress vector $P(u)\cdot n$ along solutions in the sense of transposition of \eqref{eq:Navier}, under low regularity assumptions on the data. To simplify the notation, we denote by $H^{1*}(\Gamma_0\times(0,T))$ the completion of the space $L^2(0,T;L^2(\Gamma_0))$ with respect to the norm
\[
\norm{v}_{H^{1*}(\Gamma_0\times(0,T))}:=\sup_{w\in {_0}H^1(\Gamma_0\times(0,T)),\,\norm{w}_{H^1(\Gamma_0\times(0,T))}\le 1}\left|\int^T_0\int_{\Gamma_0}w\cdot v\;d\sigma dt\right|,
\]
where ${_0}H^1(\Gamma_0\times(0,T)):=\{w\in H^1(0,T;L^2(\Gamma_0))\cap L^2(0,T;H^1(\Gamma_0)):\; w(\cdot,0)=w(\cdot,T)=0\}$. \footnote{This space is endowed with with the norm $\norm{w}_{H^1(\Gamma_0\times(0,T))}:=\left(\norm{\frac{\partial w}{\partial t}}_{L^2(0,T;L^2(\Gamma_0))}^2+\norm{w}_{L^2(0,T;H^1(\Gamma_0))}^2\right)^{1/2}$. }

\begin{theorem} \label{trace regularity from adjoint weak solution}
Under the same assumptions of Lemma \ref{existence of weak solution in the sense of adjoint-isomorphism}, we have  
\begin{equation} \label{eq:trace regularity from adjoint weak solution}
    \begin{aligned}
      \norm{P(u)\cdot n}_{H^{1*}(\Gamma_0\times(0,T))}&\le c(T)\norm{g}_{L^2(0,T;L^2(\Gamma_0))},
    \end{aligned}  
\end{equation} 
where $c(T)$ is a positive constant (depending on $T$).
\end{theorem}

\begin{proof}
Let $k\in {_0}H^1(\Gamma_0\times(0,T))$ be non-zero, and consider the solution $z$ (in the sense of transposition) of \eqref{eq:Navier} with $u_0=u_1=F=0$ and $g=k$. By Lemma \ref{Lemma: existence of weak distribution} and Theorem \ref{thm trace regularity with nonzero B.C.}, we have that 
\begin{equation}\label{eq:regularity-weak-k}
\norm{z}_{C([0,T];H^1(\Omega))}+\norm{\frac{\partial z}{\partial t}}_{C([0,T];L^2(\Omega))}
+\norm{P(z)\cdot n}_{L^2(0,T;L^2(\Gamma_0))}
\le c(T)\norm{k}_{H^1(\Gamma_0\times(0,T))}. 
\end{equation}

Assume that all data are smooth, and use \eqref{definition of weak soln} for $u$ (from Lemma \ref{existence of weak solution in the sense of adjoint-isomorphism}) with $\phi=z$ (corresponding to $\psi=0$) and for $z$ with $\phi=u$ (corresponding to $\psi=F$). Adding side by side the resulting two equations, we find 
   \begin{equation} \label{eq: definition of weak soln for lower regularity}
    \int_{0}^{T} \int_{\Gamma_{0}}k \cdot (P(u)\cdot n) \;d\sigma dt=
    -\int_{0}^{T} \int_{\Gamma_{0}}g \cdot (P(z)\cdot n) \;d\sigma dt.
  \end{equation}
Using \eqref{eq:regularity-weak-k} and H\"older's inequality, we obtain the estimate \eqref{eq:trace regularity from adjoint weak solution}. 
\end{proof}

Next we prove the existence of {\em strong solutions} of \eqref{eq:Navier} assuming more regular data. 

\begin{lemma} \label{existence of strong solution}
Assume that 
\begin{equation} \label{assumptions on strong solution}
    \begin{aligned}
        & F\in L^{1}(0,T;H^{1}(\Omega)),\quad \frac{\partial F}{\partial t}\in L^1(0,T;L^2(\Omega)), 
        \\
        &g \in L^{2}(0,T;H^{2}(\Gamma_{0})) \cap H^{2}(0,T;L^{2}(\Gamma_{0})), 
        \\ 
        & u_{0}\in H^{2} (\Omega), \quad u_{1} \in H^{1} (\Omega),  
        \end{aligned}
\end{equation}
and the following {\em compatibility conditions} are satisfied 
\begin{equation} \label{compatibility conditions on strong solution)}
    \begin{aligned}
        &g(\cdot,0) =u_{0}\quad  &&\text{on }\Gamma_{0}, \quad
        \frac{\partial g}{\partial t}(\cdot,0) =u_{1}\quad &&\text{on }\Gamma_{0}, 
        \\
        &P(u_{0}) \cdot n =0 \quad &&\text{on } \Gamma_{1}.
    \end{aligned}
\end{equation}
Then, there exists a unique solution to \eqref{eq:Navier} enjoying the following regularity properties: 
\begin{equation} \label{eq:regularity of solution}
    \begin{aligned}
     u \in C([0,T];H^{2}(\Omega)) ;\quad
      \frac{\partial u}{\partial t} \in C([0,T];H^{1}(\Omega)) ;\quad
      \frac{\partial^{2} u}{\partial t^2} &\in C([0,T];L^{2}(\Omega)) .
     \end{aligned}  
\end{equation}  
In particular, $u$ satisfies \eqref{eq:Navier} almost everywhere in space and time, and the following estimate holds 
\begin{equation} \label{eq:estimate-strong-sln}
   \begin{aligned}
        \norm{u}_{C([0,T];H^2(\Omega))}
        &+\norm{\frac{\partial u}{\partial t}}_{C([0,T];H^1(\Omega))}
        +\norm{\frac{\partial^2 u}{\partial t^2}}_{C([0,T];L^2(\Omega))}
        \\
        &\le c(T)
        \left(\norm{F}_{L^{1}(0,T;H^{1}(\Omega))}+\norm{\frac{\partial F}{\partial t}}_{L^1(0,T;L^2(\Omega))}
        +\norm{u_{0}}_{H^{2}(\Omega)}\right.
        \\
        &\qquad \qquad \qquad\qquad \qquad \qquad \quad \left.+\norm{u_{1}}_{H^1(\Omega))}
        +\norm{g}_{L^2(0,T;H^2(\Omega))}+\norm{g}_{H^2(0,T;L^2(\Omega))}\right),
   \end{aligned}
\end{equation}
with $c(T)$ is a positive constant (depending on $T$). 
\end{lemma}

\begin{proof}
From the given hypotheses and Lemma \ref{Lemma: existence of weak distribution}, there exists a unique $u$  satisfying \eqref{regularity lemma 2}, \eqref{additional regularity of lemma 2}, and 
\begin{equation} \label{weak adjoint soln-1}
\begin{split}
    \int_{0}^{T} (u,\psi)_{L^{2}(\Omega)} dt&=\langle u_1,\phi(\cdot,0)\rangle_{V^{-1}(\Omega)}
    -\left(u_0,\frac{\partial \phi}{\partial t}(\cdot,0)\right)_{L^2(\Omega)}
    \\
    &\quad-\int_{0}^{T} \int_{\Gamma_{0}}g \cdot (P(\phi)\cdot n) d\sigma dt
    +\int^T_0\langle F,\phi\rangle_{V^{-1}(\Omega)}\; dt,
\end{split}\end{equation}     
for all $\psi\in L^1(0,T;L^2(\Omega))$, with $\phi=S^{-1}\psi$ and $S$ defined in \eqref{eq:isomorphism-S}. In a similar fashion, there exists  a unique $\bar u$  satisfying \eqref{regularity lemma 2}, \eqref{additional regularity of lemma 2}, and 
\begin{equation} \label{weak adjoint soln-partial-t}
\begin{split}
\int_{0}^{T} (\bar u,\xi)_{L^{2}(\Omega)} dt&=\langle F(\cdot, 0)+\div P(u_0),\theta(\cdot,0)\rangle_{V^{-1}(\Omega)}
    -\left(u_1,\frac{\partial \theta}{\partial t}(\cdot,0)\right)_{L^2(\Omega)}
    \\
    &\quad-\int_{0}^{T} \int_{\Gamma_{0}}\frac{\partial g}{\partial t} \cdot (P(\theta)\cdot n) d\sigma dt
    +\int^T_0\left\langle \frac{\partial F}{\partial t},\theta\right\rangle_{ V^{-1}(\Omega)}\; dt,
\end{split}\end{equation}     
for all $\xi\in L^1(0,T;L^2(\Omega))$, with $\theta=S^{-1}\xi$. We note that,  by assumption and by Sobolev embedding theorem, we have that $F\in C([0,T];L^2(\Omega))$. 

Now consider $\xi=\chi\varphi\in C^\infty_c((0,T);L^2(\Omega))$ with $\chi\in C^\infty_c((0,T);\R)$ and $\varphi\in L^2(\Omega;\R^3)$ (and correspondingly $\theta=S^{-1}\xi\in C^\infty_c((0,T);H^1(\Omega))$) in \eqref{weak adjoint soln-partial-t}, and $\psi=\displaystyle \frac{\partial \xi}{\partial t}$ (and correspondingly, $\phi=\displaystyle\frac{\partial \theta}{\partial t}$) in \eqref{weak adjoint soln-1}. We find that 
\[
\int^T_0\chi'\left(u,\varphi\right)_{L^2(\Omega)}\; dt=-\int^T_0\chi(\bar u,\varphi)_{L^2(\Omega)}\; dt. 
\]
The latter implies that 
\[
\frac{\partial u}{\partial t}=\bar u\in C([0,T];H^1(\Omega)), \qquad \frac{\partial^2 u}{\partial t^2}=\frac{\partial \bar u}{\partial t}\in C([0,T];L^2(\Omega)). 
\]

It remains to show that the constructed weak solution $u$ satisfies $u\in L^\infty(0,T;H^2(\Omega))$ and all the equations in \eqref{eq:Navier} hold almost everywhere in space and time. From our assumptions on $g$, $u_0$, and $u_1$, together with the compatibility conditions \eqref{compatibility conditions on strong solution)}, we can apply Theorem \ref{th:continuity-time-L2-Hm} to $g$ (with $H=L^2(\Gamma_0)$, $V=H^2(\Gamma_0)$ and $m=2$) to find that 
\[
g\in C([0,T];H^{3/2}(\Gamma_0))
\]
and 
\[
\norm{g}_{C([0,T];H^{3/2}(\Gamma_0))}\le c_1\left(\norm{g}_{L^2(0,T;H^2(\Gamma_0))}+\norm{g}_{H^2(0,T;L^2(\Gamma_0))}
+\norm{u_0}_{H^2(\Omega)}+\norm{u_1}_{H^1(\Omega)}\right). 
\]
We can then find a lift $\hat g\in C([0,T];H^{5/2}(\Omega))$ such that $\hat g=g$ on $\Gamma_0\times (0,T)$ and $P(\hat g)\cdot n=0$ on $\Gamma_1\times(0,T)$. In addition, 
\[
\norm{\hat g}_{C([0,T];H^{5/2}(\Omega))}\le c_2\left(\norm{g}_{L^2(0,T;H^2(\Gamma_0))}+\norm{g}_{H^2(0,T;L^2(\Gamma_0))}
+\norm{u_0}_{H^2(\Omega)}+\norm{u_1}_{H^1(\Omega)}\right). 
\]
Now consider the solution $\hat u$ to the following elliptic problem (with time now considered as fixed parameter)
\begin{equation}\label{eq:elliptic}\left\{\begin{aligned}
  &\div P(\hat u)=\frac{\partial^2 u}{\partial t^2}-\div P(\hat g) - F\qquad&&\text{in }\Omega,
  \\
  &\hat u=0\qquad&&\text{on }\Gamma_0,
  \\
  &P(\hat u)\cdot n=0\qquad&&\text{on }\Gamma_1.
\end{aligned}\right.\end{equation}
By elliptic regularity (see \cite[Theorem 6.3-6 \& following remarks on p. 298]{ciarlet}), we know that 
\[
\hat u\in C([0,T];H^{2}(\Omega))
\]
and the above equations are satisfied almost everywhere in space and time. Set $w:=\hat u+\hat g\in C([0,T];H^{2}(\Omega))$ and note that it is the unique strong solution of the elliptic problem 
\[\left\{\begin{aligned}
  &\div P(w)=\frac{\partial^2 u}{\partial t^2}- F\qquad&&\text{in }\Omega\times(0,T),
  \\
  &w=g\qquad&&\text{on }\Gamma_0\times(0,T),
  \\
  &P(w)\cdot n=0\qquad&&\text{on }\Gamma_1\times(0,T).
\end{aligned}\right.\]
Multiply the first equation in the latter system by $\phi=S^{-1}\psi\in C^\infty_c([0,T);H^1_0(\Omega))$ (with $\psi\in C^\infty_c([0,T);L^2(\Omega))$ and integrate the resulting equation over $(0,T)$, using integration by parts, we find that $w$ satisfies 
\begin{multline*}
-\int^T_0\int_{\Gamma_0}g\cdot P(\phi)\cdot n\;d\sigma d t=\int^T_0(w,\psi)_{L^2(\Omega)}
\\
-\langle u_1,\phi(\cdot,0)\rangle_{V^{-1}(\Omega)}
+\left(u_0,\frac{\partial \phi}{\partial t}(\cdot,0)\right)_{L^2(\Omega)}
-\int^T_0(F,\phi)_{L^2(\Omega)}\; dt.
\end{multline*}
By uniqueness of the solution to \eqref{weak adjoint soln-1}, we conclude that $u=w$. 
Finally, the estimate \eqref{eq:estimate-strong-sln} follows from the above considerations and \eqref{eq:energy-estimate-adjoint-regular} applied to $u$ and $\bar u=\displaystyle \frac{\partial u}{\partial t}$ from \eqref{weak adjoint soln-1} and \eqref{weak adjoint soln-partial-t}, respectively. 
\end{proof}

Now we can use the hypotheses and compatibility conditions of Lemma \ref{existence of strong solution} to unveil some ``hidden regularity'' of the stress vector $P(u)\cdot n$ as in the following theorem.
 \begin{theorem} \label{th:strong trace regularity}
Under the same assumptions and compatibility conditions as in Lemma \ref{existence of strong solution} we have that the solution $u$ to \eqref{eq:Navier}, beside  \eqref{eq:regularity of solution}, also satisfies the following regularity property
\begin{equation} \label{eq:regularity of the stress vector}
    P(u) \cdot n \in L^{2}(0,T;H^{1}(\Gamma_{0})) \cap H^{1} (0,T;L^2(\Gamma_{0})).
\end{equation}
In particular, the following estimate holds 
\begin{equation} \label{eq:estimate-trace regularity-strong-sln}
   \begin{aligned}
    \norm{P(u)\cdot n}_{L^2(0,T;H^1(\Gamma_0))}&+\norm{P(u)\cdot n}_{H^1(0,T;L^2(\Gamma_0))}
    \\
    &\le {c(T)}
    \left(\norm{F}_{L^{1}(0,T;H^{1}(\Omega))}+\norm{\frac{\partial F}{\partial t}}_{L^1(0,T;L^2(\Omega))}
    +\norm{u_{0}}_{H^{2}(\Omega)}\right.
    \\
    &\qquad \qquad \qquad \qquad \qquad \left.
    +\norm{u_{1}}_{H^1(\Omega))}
    +\norm{g}_{L^2(0,T;H^2(\Omega))}+\norm{g}_{H^2(0,T;L^2(\Omega))}\right),
   \end{aligned}
\end{equation}
with { $c(T)$ a positive constant (depending on $T$)}. 
\end{theorem}

To prove Theorem \ref{th:strong trace regularity} we will need the following proposition.
\begin{proposition} \label{proposition for stress vector}
Assume that $P(u)\cdot n \in L^{2}(0,T;L^2(\Gamma_0))$ and $u=0$ on $\Gamma_0$, then  
\begin{equation} \label{corollary for stress vector}
    P(u)\in L^{2}(0,T;L^2(\Gamma_0)). 
    \end{equation}
\end{proposition}
\begin{proof}
Using \eqref{eq:div^2} and \eqref{eq:sym-nabla-u-dot-n-norm} from Appendix \ref{sec:equalities}, we can write 
    \begin{equation*}
        \left|P(u)\cdot n\right|^2=\mu^2 (|\nabla u|^{2} + 3 (\mathrm{div}u)^2)+4\mu \lambda (\div{u})^2+\lambda^2(\div{u})^2,
    \end{equation*}
implying that $\nabla u\in L^2(0,T;L^2(\Gamma_0))$ and $\div u\in L^2(0,T;L^2(\Gamma_0))$, and this proves \eqref{corollary for stress vector}.
\end{proof}

\begin{proof}[Proof of Theorem \ref{th:strong trace regularity}]
For $\alpha=1,2$, we define the linear operator 
\[
B^{(\alpha)}:\; v\in H^1(\R^3)\mapsto B^{(\alpha)}v=\sum^3_{k=1}b^\alpha_k\frac{\partial v}{\partial x_k}\in L^2(\R^3)
\]
where $b^\alpha$ is a (time-independent) smooth vector field whose components satisfy 
\begin{equation}\label{eq:tangential-vector}
b^\alpha_i(x)=\left\{\begin{aligned}
&\frac{\partial \Theta_i}{\partial y_\alpha}\qquad&&\text{if $x=\Theta(y)$ for some $y\in \mathbb{B}_{r_i}(0)$ and some $i\in\{1,\dots,m\}$},
\\
&\delta_{i\alpha}\qquad&&\text{otherwise}
\end{aligned}\right.\end{equation}
and $\Theta$ denotes the inverse of the parametrization $\Phi$ given in \eqref{eq:parametrization-around-boundary} and \eqref{eq:parametrization-around-boundary2}, when $\mathcal D=\Omega$ and $d=3$. We note that $b^\alpha\cdot n=0$ on $\partial \Omega$. 

Let $u$ be the solution to \eqref{eq:Navier} from Lemma \ref{existence of strong solution}. In particular, it satisfies $P(u)\cdot n\in L^2(0,T;L^2(\Gamma_0))$ from Theorem \ref{thm trace regularity with nonzero B.C.}. Now, we consider a partition of unity $\{\psi_i\}_{i=1}^m$ subordinated to an open cover $\{\mathbb{B}_{r_i}(x^i)\}_{i=1}^m$ of $\overline{\Omega}$ as in \eqref{eq:open-cover}, and a smooth cut-off function $\eta$ supported on an open set $\mathcal O$ such that $\overline{\mathcal C}\subset \mathcal O\subset \Omega$ and $\eta=1$ on $\overline{\mathcal C}$. Define 
\begin{equation}\label{eq:w-Balpha-u}
w:=\eta B^{(\alpha)}u=\sum^m_{i=1}\eta B^{(\alpha)}(\psi_i u). 
\end{equation}
Given the above observations, to prove that $P(u) \cdot n \in L^{2}(0,T;H^{1}(\Gamma_{0}))$, it remains to prove that 
\[
\sum^m_{i=1}\int_{\R^2}\left|B^{(\alpha)}(\psi_iP(u)\cdot n)|_{x=\Theta(y_1,y_2,0)}\right|^2\;dy_1dy_2<\infty\qquad\text{for each }\alpha=1,2. 
\]
The latter is true provided that 
\begin{equation}\label{eq: regularity of the operator}
P(w)\cdot n \in L^{2}(0,T;L^2(\Gamma_0)).
\end{equation} 
In fact, we can formally write the following for $\alpha=1,2$:  
\begin{equation} \label{eq:relation between B(P)-P(B)}
    \begin{aligned}
       &B^{(\alpha)}(P(u) \cdot n)=\sum^3_{\ell=1}b_{\ell}^{\alpha}\frac{\partial}{\partial x_\ell}(P(u) \cdot n) 
       \\
        &\quad =\sum_{\ell=1}^{3}b_{\ell}^{\alpha}\frac{\partial}{\partial x_\ell}(2\mu e_{kj}(u)+\lambda \div{u}\delta_{kj})n_j\mathsf{e}^k+P(u-\hat g)\cdot B^{(\alpha)}n 
        +P(\hat g)\cdot B^{(\alpha)}(n) 
        \\
        &\quad =P(w)\cdot n
        -(\mu \frac{\partial b_{\ell}^{\alpha}}{\partial x_j}\frac{\partial u_k}{\partial x_\ell}+\mu \frac{\partial b_\ell^\alpha}{\partial x_k}\frac{\partial u_j}{\partial x_\ell}+\lambda\frac{\partial u_n}{\partial x_\ell}\frac{\partial b_\ell^\alpha}{\partial x_n}\delta_{kj})n_j\mathsf{e}^k
        \\
        &\qquad\qquad\qquad\qquad\qquad\qquad\qquad\qquad\qquad\qquad\qquad +P(u-\hat g)\cdot B^{(\alpha)} n +P(\hat g)\cdot B^{(\alpha)}n,
    \end{aligned}
\end{equation}
where $\hat g\in L^2(0,T;H^{5/2}(\Omega))$ is a solution of \eqref{eq:divP0-extension} (see \cite[Chapter 2, Section 5, Theorem 5.4]{LM-1}). We note that the last five terms on the right-hand side of \eqref{eq:relation between B(P)-P(B)} can be controlled thanks to the regularity of $u$, $g$, and $b^\alpha_\ell$, and thanks to Proposition \ref{proposition for stress vector}.

Let us (again formally) apply $\eta B^{(\alpha)}$ to both sides of \eqref{eq:Navier} and observe that
    \begin{equation} \label{eq: navier with w}
        \frac{\partial^2w}{\partial t^2}= \div{P(w)}+\eta B^{(\alpha)}F+\Tilde{R}(u),        
    \end{equation}
where $\Tilde{R}(u)$ is given by 
     \begin{equation*}
         \begin{aligned}
           \Tilde{R}(u):= -\mu\div(B^{(\alpha)}u \otimes\nabla \eta+\nabla \eta \otimes B^{(\alpha)}u)
           -\lambda\nabla(\nabla \eta \cdot B^{(\alpha)}u)
           -P(B^{(\alpha)}u)\cdot \nabla\eta+\eta R(u)
         \end{aligned}
     \end{equation*}
and \[
    \begin{split}
          R(u)&:=-\mu\sum^3_{k=1}\left[\frac{\partial u_i}{\partial x_k}\Delta b^\alpha_k 
          +2\frac{\partial b^\alpha_k}{\partial x_j}\frac{\partial^2 u_i}{\partial x_j\partial x_k}\right]\mathsf{e}^i
          \\
          &\qquad -(\lambda+\mu)\sum^3_{k=1}\left[\frac{\partial^2b^\alpha_k}{\partial x_i\partial x_\ell}\frac{\partial u_\ell}{\partial x_k}
          +\frac{\partial b^\alpha_k}{\partial x_\ell}\frac{\partial^2u_\ell}{\partial x_i\partial x_k}
          +\frac{\partial b^\alpha_k}{\partial x_i}\frac{\partial \div(u)}{\partial x_k}\right]\mathsf{e}^i.
      \end{split}
    \]
Set $w_0:=\eta B^{(\alpha)}u_{0}\in H^{1}(\Omega)$, $w_1:=\eta B^{(\alpha)}u_{1} \in L^{2}(\Omega)$, and $g_w:=\eta B^{(\alpha)}g \in L^2(0,T;H^1(\Gamma_{0}))$. We notice that, by Theorem \ref{th:continuity-time-L2-Hm}, assumption \eqref{assumptions on strong solution}, and the compatibility conditions \eqref{compatibility conditions on strong solution)}, we have that $g\in C([0,T];H^{3/2}(\Gamma_0))$. Thus, the compatibility condition 
$g_w(\cdot,0)=w_0$ holds on $\Gamma_0$ and $P(w_0)\cdot n=0$ on $\Gamma_1$. Furthermore, by the Theorem \ref{th:intermediate-derivative}, $\displaystyle \frac{\partial g}{\partial t}\in L^2(0,T;H^1(\Gamma_0))$. So that 
\[
\frac{\partial g_w}{\partial t}\in L^2(0,T;L^2(\Gamma_0)). 
\]
Finally, by Lemma \ref{existence of strong solution}, we have that $\tilde F:=\eta B^{(\alpha)}F+\Tilde{R}(u)\in L^{2}(0,T;L^{2}(\Omega))$, and we can apply Theorem \ref{thm trace regularity with nonzero B.C.} to conclude that $w$ --defined in \eqref{eq:w-Balpha-u}-- is a weak solution of the following problem 
\begin{equation}
    \begin{aligned}
        & \frac{\partial^2 w}{\partial t^2}=\mathrm{div}P(w)+\tilde F 
            &&\text{ in }\Omega\times(0,T),
        \\
        & w(x,0)=w_0 &&\text{ in }\Omega,
        \\
        & \frac{\partial w}{\partial t}(x,0)=w_1 &&\text{ in }\Omega,
        \\
        & w=g_w &&\text{ on }\Gamma_0\times(0,T), 
        \\
        & P(w)\cdot n=0 &&\text{ on }\Gamma_1\times(0,T).   
    \end{aligned}
\end{equation}
In particular, \eqref{trace regularity for nonhomogeneous case} holds with $w$ in place of $u$ (note also that $w=\eta B^{(\alpha)}u\equiv B^{(\alpha)}u$ on $\Gamma_0$). 

The fact that $P(u)\cdot n\in H^1(0,T;L^2(\Gamma_0))$ follows immediately from the proof of Lemma \ref{existence of strong solution} (since $\bar u=\displaystyle\frac{\partial u}{\partial t}$ satisfies \eqref{weak adjoint soln-partial-t}) and from \eqref{trace regularity for nonhomogeneous case} with $u$ now replaced by $\displaystyle\frac{\partial u}{\partial t}$. 
\end{proof}

Since \eqref{eq:Navier} is linear, from Lemma \ref{existence of strong solution} and by repeatedly taking time-derivatives of \eqref{eq:Navier}, one immediately ensures the existence of solutions corresponding to data in higher regularity classes (see next theorem for the precise statement). In order to prove the corresponding sharp trace regularity of the stress vector $P(u)\cdot n$ on $\Gamma_0$, we consider the $m$-tuple $(\alpha_1,\dots,\alpha_m)$, with $m\in \N\setminus\{0\}$ and $\alpha_i\in \{1,2\}$ for $i=1,\dots,m$, and the time-independent operator
\begin{equation*}
\mathcal B^{(\alpha_1,\dots,\alpha_m)}:\; v\in H^m(\Omega)
\mapsto \mathcal B^{(\alpha_1,\dots,\alpha_m)}v:=\sum^3_{i_1,\dots,i_m=1}b^{(\alpha_1)}_{i_1}b^{(\alpha_m)}_{i_m}\frac{\partial^m v}{\partial x_{i_1}\dots\partial x_{i_m}}\;\in L^2(\Omega),
\end{equation*}
where the coefficients $b^{\alpha}_{i}$ are defined in \eqref{eq:tangential-vector}. We then follow the strategy of the proof of Theorem \ref{th:strong trace regularity}, by replacing the tangential operator $B^{(\alpha)}$ with $\mathcal B^{(\alpha_1,\dots,\alpha_m)}$, to prove the following result. 

\begin{theorem}\label{th:higher-regularity}
Let $k\in \N$, and assume that\footnote{We use the notation $H^0(\Omega)=L^2(\Omega)$. } 
\begin{equation}\label{eq:data-high-regularity}\begin{aligned}
&\frac{\partial^j F}{\partial t^j}\in L^1(0,T;H^{k-j}(\Omega))\quad\text{ for }j=0,\dots,k,
\\
& u_{0}\in H^{k+1} (\Omega), \quad u_{1} \in H^{k} (\Omega),  
\\
&g \in L^{2}(0,T;H^{k+1}(\Gamma_{0})) \cap H^{k+1}(0,T;L^{2}(\Gamma_{0})). 
\end{aligned}
\end{equation}
We suppose that the following {\em compatibility conditions} hold for each $j=0,\dots, k$ 
\[\begin{aligned}
&\frac{\partial^{j}g}{\partial t^{j}}\left.\right|_{t=0}=u^{(j)}\qquad&&\text{on }\Gamma_0,
\\
&P(u^{(j)})\cdot n=0\qquad&&\text{on }\Gamma_1,
\end{aligned}\]
where $u^{(0)}=u_0$, $u^{(1)}=u_1$, and if $j\ge 2$, $u^{(j)}\in H^{k+1-j}(\Omega)$ satisfies 
\[
\begin{aligned}
u^{(j)}=\div P(u^{(j-2)})+\frac{\partial^{j-2}F}{\partial t^{j-2}}\left.\right|_{t=0}\quad&&\text{on }\Omega. 
\end{aligned}
\]
Then, there exists a unique solution to \eqref{eq:Navier} that enjoys the following regularity properties: 
\begin{equation} \label{eq:high-interp-regularity}
\begin{aligned}
&\frac{\partial^{j} u}{\partial t^j} \in C([0,T];H^{k+1-j}(\Omega)) \qquad \text{ for each }j=0,\dots,k+1,
\\
&P(u) \cdot n \in L^{2}(0,T;H^{k}(\Gamma_{0})) \cap H^{k} (0,T;L^2(\Gamma_{0})).
\end{aligned}  
\end{equation}  
In particular, $u$ satisfies \eqref{eq:Navier} almost everywhere in space and time, and the following estimate holds 
\begin{equation} \label{eq:high-interp-estimate}
   \begin{aligned}
        \sup_{t\in [0,T]}\sum^{k+1}_{j=0}\norm{\frac{\partial^{j} u(t)}{\partial t^j}}_{H^{k+1-j}(\Omega)}
        &+\norm{P(u) \cdot n}_{L^{2}(0,T;H^{k}(\Gamma_{0}))}
        +\norm{P(u)\cdot n}_{H^{k} (0,T;L^2(\Gamma_{0}))}
        \\
        \le {c(T)}&
        \left(\sum^k_{j=0}\norm{\frac{\partial^j F}{\partial t^j}}_{L^1(0,T;H^{k-j}(\Omega))}
        +\norm{u_{0}}_{H^{k+1}(\Omega)}
        +\norm{u_{1}}_{H^{k}(\Omega)}
        \right.
        \\
        &\qquad\qquad\qquad\qquad\left. 
        +\norm{g}_{L^2(0,T;H^{k+1}(\Omega))}+\norm{g}_{H^{k+1}(0,T;L^2(\Omega))}\right),
   \end{aligned}
\end{equation}
where  $c(T)$ is a positive constant (depending on $T$). 
\end{theorem}

We conclude this section with the following theorem which asserts that the constant ``$c(T)$'' in the estimate \eqref{eq:high-interp-estimate} can be made independent of $T$. The proof exploits and extends the ideas in \cite[pp. 558--560]{raymond2014} for the same equation but with Dirichlet boundary conditions, $k=0,1,2$,  
and $F=0$.  
\begin{theorem}\label{th:constant-indep-T}
Under the same hypotheses of Theorem \ref{th:higher-regularity}, and for $k\ge 1$ assume further that 
\[
\frac{\partial^j F}{\partial t^j}\left.\right|_{t=0}\in H^{k-j-1/2}(\Omega)\quad\text{for }j=0,\dots k-1. 
\]
Then, the solution to \eqref{eq:Navier} satisfies the following estimate with a constant $c>0$ independent of $T$: 
\begin{equation} \label{eq:high-estimate-indep-T}
   \begin{aligned}
        &\sup_{t\in [0,T]}\sum^{k+1}_{j=0}\norm{\frac{\partial^{j} u(t)}{\partial t^j}}_{H^{k+1-j}(\Omega)}
        +\norm{P(u) \cdot n}_{L^{2}(0,T;H^{k}(\Gamma_{0}))}
        +\norm{P(u)\cdot n}_{H^{k} (0,T;L^2(\Gamma_{0}))}
        \\
        &\qquad \le c
        \left(\norm{F}_{L^2(0,T;H^k(\Omega))}+\norm{F}_{H^k(0,T;L^2(\Omega))}+\sum^{k-1}_{j=0}\norm{\frac{\partial^j F}{\partial t^j}\left.\right|_{t=0}}_{L^1(0,T;H^{k-j-1/2}(\Omega))}\right.
        \\
        &\quad\qquad\qquad\left.\phantom{\frac{dx}{dt}} +\norm{u_{0}}_{H^{k+1}(\Omega)}
        +\norm{u_{1}}_{H^{k}(\Omega)}\
        +\norm{g}_{L^2(0,T;H^{k+1}(\Omega))}+\norm{g}_{H^{k+1}(0,T;L^2(\Omega))}\right).
   \end{aligned}
\end{equation}
\end{theorem}

\begin{proof}
    The main idea of the proof is to re-write the unique solution $u$ of \eqref{eq:Navier} from Theorem \ref{th:higher-regularity} as $u=\hat u+\tilde u$ where $\hat u$ and $\tilde u$ are solutions to \eqref{eq:Navier} on a time-interval whose length is independent of $T$, and corresponding to a suitable choice of data. We explain this in detail next. 

    Fix $\tau\in (T,\infty)$. Under the stated hypotheses on the data and \cite[Chapter 1, Section 3., Theorem 3.2]{LM-1}, there exist 
    \[\begin{split}
        &\hat F\in L^2(0,\infty;H^k(\Omega))\cap H^k(0,\infty;L^2(\Omega)),
        \\
        &\hat g\in L^2(0,\infty;H^{k+1}(\Gamma_0))\cap H^{k+1}(0,\infty;L^2(\Gamma_0)),
    \end{split}\]
    such that 
    \[\begin{split}
       &\frac{\partial^j \hat F}{\partial t^j}\left.\right|_{t=0}=\frac{\partial^j F}{\partial t^j}\left.\right|_{t=0}\quad\text{for }j=0,\dots k-1,
       \\
       &\frac{\partial^\ell \hat g}{\partial t^\ell}\left.\right|_{t=0}=\frac{\partial^\ell g}{\partial t^\ell}\left.\right|_{t=0}\quad\text{for }\ell=0,\dots k,
    \end{split}\]
    and 
    \[\begin{split}
        &\norm{\hat F}_{L^2(0,\infty;H^k(\Omega))}+\norm{\hat F}_{H^k(0,\infty;L^2(\Omega))}\le c_1\sum^{k-1}_{j=0}\norm{\frac{\partial^j F}{\partial t^j}\left.\right|_{t=0}}_{H^{k-j-1/2}(\Omega)},
        \\
        &\norm{\hat g}_{L^2(0,\infty;H^{k+1}(\Gamma_0))}+\norm{\hat g}_{H^{k+1}(0,\infty;L^2(\Gamma_0))}
        \le c_2\left(\norm{u_0}_{H^{k+1}(\Omega)}+\norm{u_1}_{H^k(\Omega)}\right.
        \\
        &\left. \qquad \qquad\qquad \qquad\qquad \qquad\qquad \qquad\qquad \qquad\quad +\norm{F}_{L^2(0,T;H^k(\Omega))}+\norm{F}_{H^k(0,T;L^2(\Omega))}\right).
    \end{split}\]
    We notice that the functions $f:=F-\hat F\in L^2(0,T;H^k(\Omega))\cap H^k(0,T;L^2(\Omega))$ and $G:=g-\hat g\in L^2(0,T;H^{k+1}(\Gamma_0))\cap H^{k+1}(0,T;L^2(\Gamma_0))$ satisfy 
    \[\begin{split}
       &\frac{\partial^j f}{\partial t^j}\left.\right|_{t=0}=0\quad\text{for }j=0,\dots k-1,
       \\
       &\frac{\partial^\ell G}{\partial t^\ell}\left.\right|_{t=0}=0\quad\text{for }\ell=0,\dots k,
    \end{split}\]
    respectively. We then consider the extensions by zero of $f$ and $G$ for $t<0$, and denote them $\tilde F$ and $\tilde g$, respectively. In particular, they satisfy 
    \[\begin{aligned}
       &\frac{\partial^j \tilde F}{\partial t^j}=\frac{\partial^j f}{\partial t^j}\quad&&\text{for }t\in (0,T)\text{ and each }j=0,\dots k-1,
       \\
       &\frac{\partial^j \tilde F}{\partial t^j}=0\quad&&\text{for }t\le0\text{ and each }j=0,\dots k-1,
       \\
       &\frac{\partial^\ell \tilde g}{\partial t^\ell}=\frac{\partial^\ell G}{\partial t^\ell}\quad&&\text{for }t\in (0,T)\text{ and each }\ell=0,\dots k,
       \\
       &\frac{\partial^\ell \tilde g}{\partial t^\ell}=0\quad&&\text{for }t\le0\text{ and each }\ell=0,\dots k,
    \end{aligned}\]
    and 
    \[\begin{split}
        \norm{\tilde F}_{L^2(-\infty,T;H^k(\Omega))}&+\norm{\hat F}_{H^k(-\infty,T;L^2(\Omega))}
        \\
        &\le c_3\left(\norm{F}_{L^2(0,T;H^k(\Omega))}+\norm{F}_{H^k(0,T;L^2(\Omega))}+\sum^{k-1}_{j=0}\norm{\frac{\partial^j F}{\partial t^j}\left.\right|_{t=0}}_{H^{k-j-1/2}(\Omega)}\right),
        \\
        \norm{\tilde g}_{L^2(-\infty,T;H^{k+1}(\Gamma_0))}&+\norm{\hat g}_{H^{k+1}(-\infty,T;L^2(\Gamma_0))}
        \\
        &\le c_4\left(\norm{u_0}_{H^{k+1}(\Omega)}+\norm{u_1}_{H^k(\Omega)}
        +\norm{F}_{L^2(0,T;H^k(\Omega))}+\norm{F}_{H^k(0,T;L^2(\Omega))}
        \right.
        \\
        &\left. \qquad \qquad\qquad \qquad\qquad \qquad\qquad 
        +\norm{g}_{L^2(0,T;H^{k+1}(\Gamma_0))}+\norm{g}_{H^{k+1}(0,T;L^2(\Gamma_0))}\right).
    \end{split}\]

    Now consider the initial-boundary value problems: 
    \begin{equation}\label{eq:Navier-hat}
    \begin{aligned}
            &\frac{\partial^{2}\hat u}{\partial t^2} =\mathrm{div} P(\hat u) + \hat F  &&\qquad \text{in } \Omega \times (0,\tau), \\
            &\hat u(\cdot,0) =u_{0} &&\qquad \text{in } \Omega ,\\
            &\frac{\partial \hat u}{\partial t}(\cdot,0) = u_{1} &&\qquad  \text{in } \Omega, \\
            &\hat u =\hat g && \qquad \text{on } \Gamma_{0} \times (0,\tau),\\
            & P(\hat u)\cdot n =0 && \qquad \text{on } \Gamma_{1} \times (0,\tau). 
    \end{aligned}
    \end{equation}
    and 
    \begin{equation}\label{eq:Navier-tilde}
    \begin{aligned}
            &\frac{\partial^{2}\tilde u}{\partial t^2} =\mathrm{div} P(\tilde u) + \tilde F(T-t)  &&\qquad \text{in } \Omega \times (0,\tau), \\
            &\tilde u(\cdot,0) =0 &&\qquad \text{in } \Omega ,\\
            &\frac{\partial \tilde u}{\partial t}(\cdot,0) = 0 &&\qquad  \text{in } \Omega, \\
            &\tilde u =\tilde g(T-t)  && \qquad \text{on } \Gamma_{0} \times (0,\tau),\\
            & P(\hat u)\cdot n =0 && \qquad \text{on } \Gamma_{1} \times (0,\tau). 
    \end{aligned}
    \end{equation}
    We then apply Theorem \ref{th:higher-regularity} to $\hat u$ and $\tilde u$, respectively. Finally, we notice that $\tilde u(t)=0$ for $t\in [T,\tau)$, and the original solution $u$ of \eqref{eq:Navier} is indeed given by 
    \[
    u=\hat u+\tilde u \qquad\text{for }t\in [0,T]. 
    \]
    By triangle inequality and \eqref{eq:high-interp-estimate} applied to $\hat u$ and $\tilde u$ (with $T$ replaced by $\tau$), we find an estimate for $u$ whose constant is of the form ``$c(\tau)$''. Since $\tau$ is arbitrary, we then find \eqref{eq:high-estimate-indep-T}. 
\end{proof}

At last, we note that, by interpolation, the above result continue to hold with the Sobolev exponent $m$ replaced by a real Sobolev exponent $s\ge 1$ and with the time derivatives of integer order possibly replaced by time-derivatives of fractional order. 

\begin{appendices}

\section{Useful equalities}\label{sec:equalities}

In this appendix, we provide the proofs of some useful equalities which are used throughout the paper. Without loss of generality, we assume that the vector fields are written in a Cartesian coordinate system with associated orthonormal basis $\{\e_1,\e_2,\e_3\}$. The symbol ``$\otimes$'' denotes the dyadic product of vectors. 

\begin{lemma}
Let $\varphi=\varphi_i\e_i \in H^1(\Omega)$ be such that $\varphi=0$ on $\Gamma_0$. Then the following equalities hold on $\Gamma_0$
\begin{align}\label{eq:nabla-u-normal}
\nabla \varphi &=\frac{\partial \varphi_i}{\partial x_k}n_kn_j \e_i\otimes\e_j,
\\
\label{eq:nabla-u-dot-n-norm}
|\nabla \varphi|&=|\nabla \varphi\cdot n|,
\\
\label{eq:div^2}
\div(\varphi)\,n\cdot\nabla \varphi\cdot n&=(\div \varphi)^2,
\\
\label{eq:sym-nabla-u-norm}
2|e(\varphi)|^2&=|\nabla \varphi|^2+(\div(\varphi))^2,
\\
\label{eq:sym-nabla-u-dot-n-norm}
4(e(\varphi)\cdot n)\cdot (e(\varphi)\cdot n)&=|\nabla \varphi|^{2} + 3 (\mathrm{div}\varphi)^2.
\end{align}
\end{lemma}

\begin{proof}[Proof of \eqref{eq:nabla-u-normal}.] 
Using the component-basis representation of the gradient of a tensor field, we have that 
\[
\nabla \varphi = \frac{\partial \varphi_i}{\partial x_j}\e_i\otimes \e_j
=\e_i\otimes \left(\frac{\partial \varphi_i}{\partial x_j}\e_j\right)=\e_i\otimes \nabla \varphi_i
\]
Note that $\varphi=0$ on $\Gamma_{0}$ is equivalent to $\varphi_i=0$ on $\Gamma_{0}$ for all $i=1,2,3$. Thus, $\nabla \varphi_i=(\nabla\varphi_i\cdot n)n$ on $\Gamma_{0}$ for all $i=1,2,3$. As a consequence, 
\[
\nabla\varphi=\e_i\otimes [(\nabla \varphi_i\cdot n)n]=(\nabla \varphi_i\cdot n)\e_i\otimes n
=\frac{\partial \varphi_i}{\partial x_k}n_k\e_i\otimes(n_j\e_j)
=\frac{\partial \varphi_i}{\partial x_k}n_kn_j\e_i\otimes \e_j. 
\]
\end{proof}

\begin{proof}[Proof of \eqref{eq:nabla-u-dot-n-norm}.]
Using the definition of norm for tensor fields and \eqref{eq:nabla-u-normal}, we have that 
\[
|\nabla \varphi|^2=\nabla \varphi\;\colon\nabla \varphi=\frac{\partial \varphi_i}{\partial x_k}n_kn_j\frac{\partial \varphi_i}{\partial x_\ell}n_\ell n_j
=\frac{\partial \varphi_i}{\partial x_k}n_k\frac{\partial \varphi_i}{\partial x_\ell}n_\ell
=(\nabla \varphi_i\cdot n)(\nabla \varphi_i\cdot n).
\]
On the other side, 
\[
|\nabla\varphi\cdot n|^2=|[(\nabla \varphi_i\cdot n)\e_i\otimes n]\cdot n|^2=|(\nabla \varphi_i\cdot n)\e_i|^2=(\nabla \varphi_i\cdot n)(\nabla \varphi_i\cdot n).
\]
\end{proof}

\begin{proof}[Proof of \eqref{eq:div^2}.]
Using again \eqref{eq:nabla-u-normal}, we obtain that 
\begin{equation*} \label{Proof of A3}
    \begin{aligned}
        \div(\varphi)\,n\cdot\nabla \varphi\cdot n&=\div(\varphi)\left(n_{i} \frac{\partial \varphi_{i}}{\partial x_{j}}n_{j}\right)
        =\div(\varphi)\frac{\partial \varphi_{i}}{\partial x_{j}}n_{i}n_{j}
        =\div(\varphi)\;\mathrm{tr}(\nabla \varphi)
        =(\div (\varphi))^{2}.
      \end{aligned}
  \end{equation*}
\end{proof}

\begin{proof}[Proof of \eqref{eq:sym-nabla-u-norm}.]
As a consequence of \eqref{eq:nabla-u-normal}, since $\{\e_i\otimes\e_j:\; i,j=1,2,3\}$ forms a basis, we have that 
\[
\frac{\partial \varphi_i}{\partial x_j}
=\frac{\partial \varphi_i}{\partial x_k}n_kn_j    \qquad\text{ for all }i,j=1,2,3. 
\]
Using this fact, we get that 
\begin{equation*} \label{proof of A4}
    \begin{aligned}
        4(e(\varphi) \colon e(\varphi))&=(\nabla \varphi +(\nabla \varphi)^{T}) \colon (\nabla \varphi +(\nabla \varphi)^{T})
        \\
        &=\frac{\partial \varphi_{i}}{\partial x_{j}}\frac{\partial \varphi_{i}}{\partial x_{j}}+\frac{\partial \varphi_{j}}{\partial x_{i}}\frac{\partial \varphi_{j}}{\partial x_{i}}
        +2\left(\frac{\partial \varphi_i}{\partial x_k}n_kn_i\right)\left(\frac{\partial \varphi_j}{\partial x_\ell}n_\ell n_j\right)     
        \\
        &=\frac{\partial \varphi_{i}}{\partial x_{j}}\frac{\partial \varphi_{i}}{\partial x_{j}}+\frac{\partial \varphi_{j}}{\partial x_{i}}\frac{\partial \varphi_{j}}{\partial x_{i}}+2
          \left(\mathrm{tr}(\nabla \varphi)\right)^2        
        \\
        &=2|\nabla \varphi^{2}|+2 (\div\varphi)^{2}.
    \end{aligned}
\end{equation*}
\end{proof}

\begin{proof}[Proof of \eqref{eq:sym-nabla-u-dot-n-norm}.]
Using the definition $e(\varphi)=\frac 12 (\nabla \varphi+(\nabla \varphi)^T)$ and property \eqref{eq:nabla-u-dot-n-norm}, we find that 
\begin{equation} \label{proof of A5}
    \begin{aligned}
        4(e(\varphi)\cdot n)\cdot (e(\varphi)\cdot n)&=(\nabla \varphi \cdot n+(\nabla \varphi)^{T}\cdot n)
           \cdot(\nabla \varphi \cdot n+(\nabla \varphi)^{T}\cdot n)
        \\
        &=|\nabla \varphi|^{2}+2((\nabla \varphi)^T\cdot n)\cdot (\nabla \varphi \cdot n)+|(\nabla \varphi)^{T} \cdot n|^2.
    \end{aligned}
\end{equation}
Let us look at the third term on the right hand side of \eqref{proof of A5}, we then use \eqref{eq:nabla-u-normal} to get 
\begin{equation} \label{eq:3rd term of A5}
    \begin{aligned}
        |(\nabla \varphi)^{T} \cdot n|^2 &=\left(\frac{\partial \varphi_{i}}{\partial x_{l}}n_{l}n_{i}\right)\left(\frac{\partial \varphi_{k}}{\partial x_{m}}n_{m}n_{k}\right)
        =(\div{\varphi})^{2}.
    \end{aligned}
\end{equation}
For what concerns the middle term on the right-hand side of \eqref{proof of A5}, using again \eqref{eq:nabla-u-normal}, we observe that 
\begin{equation} \label{eq:2nd term of A5}
       2((\nabla \varphi)^T\cdot n)\cdot(\nabla \varphi \cdot n)
       =2\left(\frac{\partial \varphi_j}{\partial x_k}n_kn_j\right)
        \left(\frac{\partial\varphi_i}{\partial x_\ell}n_\ell n_i\right)
       =2(\mathrm{div}\varphi)^2.
\end{equation}
Replacing \eqref{eq:3rd term of A5} and \eqref{eq:2nd term of A5} in \eqref{proof of A5}, \eqref{eq:sym-nabla-u-dot-n-norm} follows immediately. 
\end{proof}

\section{Regularity properties of Hilbert spaced-valued functions}\label{sec:equalities2}
Let $H$ and $V$ be two Hilbert spaces with $V$ dense in $H$ and $V\hookrightarrow H$. For $-\infty\le a<b\le +\infty$ and $m\in \N\setminus\{0\}$, we define 
\[
Z^m(a,b):=\left\{f\in L^2(a,b;V):\; \frac{d^m f}{d t^m}\in L^2(a,b;H)\right\},
\]
with the time-derivatives taken in the sense of distributions (see \cite[Definition 8.49]{leoniSobolev}). The space $Z^m(a,b)$ is endowed with the norm 
\[
\norm{f}_{Z^m(a,b)}:=\left(\norm{f}^2_{L^2(a,b;V)}+\norm{\frac{d^m f}{d t^m}}^2_{L^2(a,b;H)}\right)^{1/2}. 
\]
The following theorem has been proved in \cite{LM-1} and it also known as {\em Lions–Magenes lemma}. 
\begin{theorem}(\cite[Chapter 1, Theorem 3.1.]{LM-1})\label{th:embedding-R}
If $f\in Z^m(-\infty,\infty)$, then 
\[
\frac{d^j f}{dt^j}\in C_b\left(\R;[H,V]_{1-\frac{j+1/2}{m}}\right)
\qquad\text{for each $j=0,1,\dots,m-1$,}
\]
and there exists a constant $k_1>0$ such that  
\begin{equation}\label{eq:continuity-R}
\sup_{t\in \R}\norm{\frac{d^jf(t)}{dt^j}}_{[H,V]_{1-\frac{2j+1}{2m}}}\le k_1\norm{f}_{Z^m(-\infty,\infty)}. 
\end{equation}
\end{theorem}
Objective of this section is to extend (in an appropriate way) the previous results to Hilbert spaced-valued functions in $Z^m(0,\tau)$ with $\tau\in (0,\infty]$.  
In particular, we would like to obtain estimates with constants independent of $\tau$. 

By \cite[Chapter 1, Theorem 2.2.]{LM-1}, for $f\in Z^m(0,\infty)$, there exists an {\em extension
by reflection} $\hat f\in Z^m(-\infty,\infty)$ such that $\hat f=f$ a.e. on $(0,\infty)$. In addition, by \cite[Chapter 1, equation (2.24)]{LM-1}, there exists a constant $k_2>0$ such that 
\[
\norm{\hat f}_{Z^m(-\infty,\infty)}\le k_2\norm{f}_{Z^m(0,\infty)}. 
\]
From the previous theorem applied to the extension $\hat f$, we obtain the following result. 

\begin{theorem}\label{th:continuity-zero-infty}
If $f\in Z^m(0,\infty)$, then 
\[
\frac{d^j f}{dt^j}\in C_b\left([0,\infty);[H,V]_{1-\frac{2j+1}{2m}}\right)
\qquad\text{for each $j=0,1,\dots,m-1$,}
\]
and there exists a constant $k_3>0$ such that  
\begin{equation}\label{eq:continuity-zero-infty}
\sup_{t\in [0,\infty)}\norm{\frac{d^jf(t)}{dt^j}}_{[H,V]_{1-\frac{2j+1}{2m}}}\le k_3\norm{f}_{Z^m(0,\infty)}. 
\end{equation}
\end{theorem}

Next, we consider the space 
\[
{_0}Z^m(0,\tau):=\left\{f\in Z^m(0,\tau):\; 
\frac{d^j f}{d t^j}(0)=0\in H\text{ for each }j=0,1,\dots,m-1\right\}. 
\]

\begin{theorem}\label{th:continuity-zero-time-trace}
If $f\in {_0}Z^m(0,\tau)$, then 
\[
\frac{d^j f}{dt^j}\in C\left([0,\tau];[H,V]_{1-\frac{j+1/2}{m}}\right)
\qquad\text{for each $j=0,1,\dots,m-1$,}
\]
and there exists a constant $k_4>0$ (independent of $\tau$) such that  
\begin{equation}\label{eq:continuity-time-L2-Hm-zero-time-trace}
\sup_{t\in [0,\tau]}\norm{\frac{d^jf(t)}{dt^j}}_{[H,V]_{1-\frac{2j+1}{2m}}}\le k_4\norm{f}_{Z^m(0,\tau)}. 
\end{equation}
\end{theorem}

\begin{proof}
    In the case $m=1$, the same result can be deduced from \cite[Theorem 4.2]{meyries}. We present here a simpler proof. Let $f\in {_0}Z^1(0,\tau)$ and consider the extension 
    \[
    \tilde f:\;t\in (-\infty,\infty)\mapsto\hat f(t):=
    \left\{\begin{aligned}
        &0\qquad&&\text{for }t\le 0,
        \\
        &f(t)\qquad&&\text{for }t\in(0,\tau),
        \\
        &f(2\tau-t)\qquad&&\text{for }t\in(\tau,2\tau),
        \\
        &0\qquad&&\text{for }t\ge 2\tau. 
    \end{aligned}\right.
    \]
    We notice that $\tilde f\in Z^1(-\infty,\infty)$ and $\norm{\tilde f}_{Z^1(-\infty,\infty)}=\norm{f}_{Z^1(0,\tau)}$. We then apply Theorem \ref{th:embedding-R}, to find that 
    $f\in C\left([0,\tau];[H,V]_{\frac 12}\right)$ and
    \[
    \sup_{t\in [0,\tau]}\norm{f(t)}_{[H,V]_{\frac 12}}\leq \sup_{t\in \R}\norm{\hat f(t)}_{[H,V]_{\frac 12}}\le k_1
    \norm{f}_{Z^1(0,\tau)}. 
    \]
    Now consider $m\ge 2$. Let $f\in {_0}Z^m(0,\tau)$ and let us build a suitable extension $\hat f$ of $f$ in $Z^m(-\infty,\infty)$. To this aim, let us consider the $m$-tuple of real numbers $(\alpha_1,\alpha_2,\dots,\alpha_m)$ satisfying the following non-homogeneous linear system 
    \[
    \sum^m_{k=1}(-1)^jk^j\alpha_k=1\qquad\text{for }j=0,\dots,m-1.
    \]
    One can easily verify that the above linear system admits in fact a unique solution $(\alpha_1,\alpha_2,\dots,\alpha_m)\in \R^m$. Then, consider the extension $\tilde f:\; t\in \R\mapsto \hat f(t)\in H$ defined by 
    \begin{equation}\label{eq:extension-zero-t0}
    \tilde f(t):=\left\{\begin{aligned}
            &0 \qquad &&\text{for }t\le0,
            \\
            &f(t)\qquad &&\text{for }t\in(0,\tau),
            \\
            &\sum^m_{k=1}\alpha_k f((k+1)\tau-kt)\qquad &&\text{for }t\in\left(\tau,\frac{m+1}{m}\tau\right),
            \\
            &\sum^{m-\ell-1}_{k=1}\alpha_k f((k+1)\tau-kt)\qquad &&\text{for }t\in\left(\frac{m+1-\ell}{m-\ell}\tau,\frac{m-\ell}{m-\ell-1}\tau\right)
            \\
            &\qquad &&\qquad\text{and }\ell=0,\dots,m-2,
            \\
            &0\qquad &&\text{for }t\ge 2\tau.
        \end{aligned}\right.
    \end{equation}
    We notice that $\tilde f\in Z^m(-\infty,\infty)$ and $\norm{\tilde f}_{Z^m(-\infty,\infty)}=\norm{f}_{Z^m(0,\tau)}$. The claim then follows again from Theorem \ref{th:embedding-R}. 
\end{proof}

We conclude this appendix with the following important results. 
\begin{theorem}\label{th:continuity-time-L2-Hm}
If $f\in Z^m(0,\tau)$ satisfies that 
\begin{equation}\label{eq:trace-zero-4lift}
\frac{d^jf(0)}{dt^j}\in [H,V]_{1-\frac{2j+1}{2m}}\qquad\text{for each }j=0,1,\dots,m-1,
\end{equation}
then 
\begin{equation}\label{eq:continuous-derivative-interpolation}
\frac{d^j f}{dt^j}\in C\left([0,\tau];[H,V]_{1-\frac{2j+1}{2m}}\right)
\qquad\text{for each $j=0,1,\dots,m-1$,}
\end{equation}
and there exists a constant $k_5>0$ (independent of $\tau$) such that  
\begin{equation}\label{eq:continuity-time-L2-Hm}
\sup_{t\in [0,\tau]}\norm{\frac{d^jf(t)}{dt^j}}_{[H,V]_{1-\frac{2j+1}{2m}}}\le k_5\left(\norm{f}_{Z^m(0,\tau)}+\sum^{m-1}_{j=0}\norm{\frac{d^jf(0)}{dt^j}}_{[H,V]_{1-\frac{2j+1}{2m}}}\right). 
\end{equation}
\end{theorem}

\begin{proof}
Using the trace theorem \cite[Chapter 1, Section 3, Theorem 3.2.]{LM-1} (see also \cite[Chapter 1, Section 3, Remark 3.3.]{LM-1}), there exists $u\in Z^m(0,\infty)$ such that 
\[
\frac{d^ju(0)}{dt^j}=\frac{d^jf(0)}{dt^j}\qquad\text{for each }j\in[0,m-1],
\]
and 
\begin{equation}\label{eq:surjectivity-trace}
\norm{u}_{Z^m(0,\infty)}\le c_1\sum^{m-1}_{k=1}\norm{\frac{d^jf(0)}{dt^j}}_{[H,V]_{1-\frac{2j+1}{2m}}}. 
\end{equation}
Let us consider $\hat u:=u|_{[0,\tau]}\in Z^m(0,\tau)$, and $\hat f:=\hat u-f$. Then, $\hat f\in {_0}Z^m(0,\tau)$, and by Theorems \ref{th:continuity-zero-time-trace} and \ref{th:continuity-zero-infty}, we conclude that 
\[
\frac{d^j f}{dt^j}=\frac{d^j \hat f}{dt^j}+\frac{d^j \hat u}{dt^j}\in C\left([0,\tau];[H,V]_{1-\frac{j+1/2}{m}}\right)
\qquad\text{for each $j=0,1,\dots,m-1$.}
\]
In addition, \eqref{eq:continuity-time-L2-Hm} immediately follows by triangle inequality, \eqref{eq:continuity-time-L2-Hm-zero-time-trace}, \eqref{eq:continuity-zero-infty}, and \eqref{eq:surjectivity-trace}. 
\end{proof}

For  $f\in Z^m(0,\tau)$, the property \eqref{eq:continuous-derivative-interpolation} continues to hold without the assumption \eqref{eq:trace-zero-4lift}. However, the estimate \eqref{eq:continuity-time-L2-Hm} should be modified to take into account a possible dependence of the constant $k_5$ on $\tau$. In fact, following the proof of \cite[Chapter 1, Theorem 3.1.]{LM-1}, one can build an extension of $f$ from $\R$ to $V$, whose norm would now depend on $\tau$ (since we can no longer use the lifting of the condition $f(0)$). Let us investigate how the estimate \eqref{eq:continuity-time-L2-Hm} changes in this scenario. Following \cite[Chapter 1, Theorem 3.1.]{LM-1}, we just know that for $f\in Z^m(0,\tau)$, there exists a positive constant $c_\tau$, depending on $\tau$, such that 
\begin{equation}\label{eq:continuity-time-L2-Hm-tau}\tag{\ref{eq:continuity-time-L2-Hm}*}
\sup_{t\in [0,\tau]}\norm{\frac{d^jf(t)}{dt^j}}_{[H,V]_{1-\frac{2j+1}{2m}}}\le c_\tau\norm{f}_{Z^m(0,\tau)}
\end{equation}
for each $j=0,1,\dots,m-1$. Consider the function $f^*:\;t^*\in (0,1)\mapsto f^*(t^*):=f(t^*\tau)\in V$. We notice that 
\[
\norm{f^*}_{Z^m(0,1)}=\left(\frac 1\tau \norm{f}^2_{L^2(0,\tau;V)}+\tau^{2m-1}\norm{\frac{d^m f}{dt^m}}^2_{L^2(0,\tau;H)}\right)^{1/2}.
\]
We apply \eqref{eq:continuity-time-L2-Hm-tau} to $f^*$ with the constant now $c_1$ (obviously independent of $\tau$), and conclude that 
for $j=0,1,\dots,m-1$: 
\begin{equation}\label{eq:continuity-time-L2-Hm-scaled}\begin{split}
\sup_{t\in [0,\tau]}\norm{\frac{d^jf(t)}{dt^j}}_{[H,V]_{1-\frac{2j+1}{2m}}}
&=\tau^{-j}\sup_{t^*\in [0,1]}\norm{\frac{d^jf^*(t^*)}{d(t^*)^j}}_{[H,V]_{1-\frac{2j+1}{2m}}}
\\
&\le c_1\tau^{-j-\frac 12}\left(\norm{f}^2_{L^2(0,\tau;V)}+\tau^{2m}\norm{\frac{d^m f}{dt^m}}^2_{L^2(0,\tau;H)}\right)^{1/2}.
\end{split}\end{equation}

We use the latter estimate to prove the following corollary which is a revised version of Lions-Magenes lemma for slightly more regular functions. 

\begin{corollary}\label{cor:continuity-time-Linfty-Hm}
    If $f\in L^\infty(0,\tau;V)$ and $\displaystyle\frac{d^mf}{dt^m}\in L^\infty(0,\tau;H)$, then \eqref{eq:continuous-derivative-interpolation} still holds, and there exists a positive constant $k_5^*$, independent of $\tau$, such that 
    \begin{equation}\label{eq:continuity-time-Linfty-Hm}
    \sup_{t\in [0,\tau]}\norm{\frac{d^jf(t)}{dt^j}}_{[H,V]_{1-\frac{2j+1}{2m}}}
    \le k_5^*\tau^{-j}\left(\norm{f}^2_{L^\infty(0,\tau;V)}+\tau^{2m}\norm{\frac{d^m f}{dt^m}}^2_{L^\infty(0,\tau;H)}\right)^{1/2}
    \end{equation}
    for $j=0,1,\dots,m-1$. 
\end{corollary}
\begin{proof}
Under the stated assumptions, $f\in Z^m(0,\tau)$ and 
\[
\begin{split}
    \norm{f}_{L^2(0,\tau;V)}&\le \sqrt{\tau}\norm{f}_{L^\infty(0,\tau;V)},
    \\
    \norm{\frac{d^m f}{dt^m}}_{L^2(0,\tau;H)}&\le \sqrt{\tau}\norm{\frac{d^m f}{dt^m}}_{L^\infty(0,\tau;H)}. 
\end{split}
\]
Using the latter on the right-hand side of \eqref{eq:continuity-time-L2-Hm-scaled}, we obtain \eqref{eq:continuity-time-Linfty-Hm}. 
\end{proof}

Using the same lift $\hat u$ of the condition at $t=0$ as in Theorem \ref{th:continuity-time-L2-Hm}, plus the extension $\tilde f$ of $\hat f$ defined in \eqref{eq:extension-zero-t0} (with $f$ now replaced by $\hat f$), and the {\em intermediate derivative theorem} \cite[Chapter 1, Section 2.2, Theorem 2.3.]{LM-1}, we have the following additional result.
\begin{theorem}\label{th:intermediate-derivative}
If $f\in Z^m(0,\tau)$ satisfies that 
\[
\frac{d^jf(0)}{dt^j}\in [H,V]_{1-\frac{2j+1}{2m}}\qquad\text{for each }j=0,1,\dots,m-1,
\]
then 
\begin{equation}\label{eq:intermediate-derivative}
\frac{d^j f}{dt^j}\in L^2\left(0,\tau;[H,V]_{1-\frac{j}{m}}\right)
\qquad\text{for each $j=0,1,\dots,m-1$,}
\end{equation}
and there exists a constant $k_6>0$ (independent of $\tau$) such that for each $j=0,1,\dots,m-1$:
\begin{equation}\label{eq:interpolation-L2-Hm}
\norm{\frac{d^jf(t)}{dt^j}}_{L^2\left(0,\tau;[H,V]_{1-\frac{j}{m}}\right)}\le k_6\left(\norm{f}_{Z^m(0,\tau)}+\sum^{m-1}_{j=0}\norm{\frac{d^jf(0)}{dt^j}}_{[H,V]_{1-\frac{2j+1}{2m}}}\right). 
\end{equation}
\end{theorem}

With the same reasoning leading to the estimate \eqref{eq:continuity-time-L2-Hm-scaled}, we also have the following result. 
\begin{corollary}\label{cor:intermediate-derivative-notrace}
If $f\in Z^m(0,\tau)$, then \eqref{eq:intermediate-derivative} holds, and there exists a constant $k_6^*>0$ (independent of $\tau$) such that for each $j=0,1,\dots,m-1$: 
\begin{equation}\label{eq:interpolation-L2-Hm-notrace}
\norm{\frac{d^jf(t)}{dt^j}}_{L^2\left(0,\tau;[H,V]_{1-\frac{j}{m}}\right)}\le k_6^*\tau^{1-j}\left(\norm{f}^2_{L^2(0,\tau;V)}+\tau^{2m}\norm{\frac{d^m f}{dt^m}}^2_{L^2(0,\tau;H)}\right)^{1/2}. 
\end{equation}
\end{corollary}

\end{appendices}

\section*{Acknowledgements}
%

Giusy Mazzone is a member of ``Gruppo Nazionale per l'Analisi Matematica, la Probabilit\`a e le loro Applicazioni'' (GNAMPA) of Istituto Nazionale di Alta Matematica Francesco Severi (INdAM).

\section*{Declarations}




\subsection*{Funding}

The authors gratefully acknowledge the support of the Natural Sciences and Engineering Research Council of Canada (NSERC) through the NSERC Discovery Grants RGPIN-2021-03129.

\subsection*{Conflict of interest/Competing interests}

Not applicable. 

\subsection*{Ethics approval and consent to participate}

Not applicable. 

\subsection*{Consent for publication}

Not applicable. 

\subsection*{Data availability}

Not applicable. 

\subsection*{Materials availability}

Not applicable. 

\subsection*{Code availability}

Not applicable. 

\subsection*{Author contribution}

All authors contributed equally to this work.

\end{document}